\documentclass{article}

\usepackage[english]{babel}
\usepackage[latin1]{inputenc}
\usepackage[T1]{fontenc}
\usepackage{amssymb,amsmath,amstext,amsfonts,amsthm,braket}
\usepackage{mathtools}
\usepackage{verbatim}
\usepackage{relsize}
\usepackage[colorlinks=true, urlcolor=blue, linkcolor=blue, citecolor=blue]{hyperref}
\usepackage{fancyvrb}
\usepackage{tikz-cd}
\usepackage[percent]{overpic}
\usepackage{comment}
\usepackage{algorithm}
\usepackage{algpseudocode}
\usepackage{todonotes}
\usepackage{xcolor}
\usepackage{cleveref}
\usepackage{gensymb}
\usepackage{algorithm}
\usepackage{algpseudocode}
\usepackage{optidef} 
\usepackage{kbordermatrix}
\usepackage{subcaption}

\usepackage{makecell}

\usepackage[normalem]{ulem}

\usepackage{enumitem}

\usepackage[letterpaper,top=2cm,bottom=2cm,left=3cm,right=3cm,marginparwidth=1.75cm]{geometry}

\numberwithin{equation}{section}

\theoremstyle{plain}
\newtheorem*{theorem*}{Theorem}
\newtheorem{theorem}{Theorem}
\numberwithin{theorem}{section}
\newtheorem{proposition}[theorem]{Proposition}
\newtheorem{lemma}[theorem]{Lemma}
\newtheorem{corollary}[theorem]{Corollary}

\newtheorem{question}{Question}

\theoremstyle{definition}
\newtheorem{definition}[theorem]{Definition}
\newtheorem{remark}[theorem]{Remark}
\newtheorem{example}[theorem]{Example}

\DeclareMathOperator{\rank}{rank}

\title{Geometry of low nonnegative rank matrix completion}
\author{Kaie Kubjas and Lilja Mets\"alampi}

\begin{document}
\maketitle

\begin{abstract}
We study completion of partial matrices with nonnegative entries to matrices of nonnegative rank at most $r$ for some $r \in \mathbb{N}$. Most of our results are for $r \leq 3$. We show that a partial matrix with nonnegative entries has a nonnegative rank-1 completion if and only if it has a rank-1 completion. This is not true in general when $r \geq 2$. For $3 \times 3$ matrices, we characterize all the patterns of observed entries when having a rank-2 completion is equivalent to having a nonnegative rank-2 completion. If a partial matrix with nonnegative entries has a rank-$r$ completion that is nonnegative, where $r \in \{1,2\}$, then it has a nonnegative rank-$r$ completion. We will demonstrate examples for $r=3$ where this is not true. We do this by introducing a geometric characterization for nonnegative rank-$r$ completion employing families of nested polytopes which generalizes the geometric characterization for nonnegative rank introduced by Cohen and Rothblum (1993).
\end{abstract}

\section{Introduction}\label{sec:introduction}

Let $M \in \mathbb{R}_{\geq 0}^{p \times q}$ be a real matrix with nonnegative entries. 
A size-$r$ \textit{nonnegative factorization} of $M$ is $M=AB$ where $A \in \mathbb{R}_{\geq 0}^{p \times r}$ and $B \in \mathbb{R}_{\geq 0}^{r \times q}$.
The \textit{nonnegative rank} of $M$ is the smallest $r \in \mathbb{N}$ for which there exists a size-$r$ nonnegative factorization of $M$. It is denoted by $\rank_+(M)$. The goal of this paper is to study completion of matrices with missing entries (partial matrices) to matrices of nonnegative rank at most $r$, for a fixed $r \in \mathbb{N}$.

Matrix completion problems come up in applications with incomplete data. 
Examples of such applications include recommender systems, which can be exemplified by the famous Netflix problem \cite{koren2009matrix}; user data is stored in a matrix with rows corresponding to each user and columns corresponding to titles.
The observed entries correspond to a user's rating of titles. 
The problem is then to recommend titles to users based on the observed data in the matrix.
Other applications with incomplete data arise in computer vision \cite{tomasi1992shape}, image compression and restoration \cite{linonlocal}, wireless communication \cite{limillimeter}, and many others; see \cite{nguyen2019low} and the references therein for more examples.
In many examples, the low-rank assumption is natural \cite{udellbigdata}. 
Furthermore, in many applications, the data is naturally nonnegative, and the nonnegative rank has useful interpretations in terms of the structure of the data, and therefore it is more natural to ask for the existence of low nonnegative rank completions of matrices. 
For example, the user ratings in the Netflix problem and the data (pixels) in image reconstruction are nonnegative.

Some questions that one can ask about low nonnegative rank matrix completion are: Given a partial matrix, does it have a nonnegative rank-$r$ completion? If yes, is the nonnegative rank-$r$ completion unique?  If a partial matrix has a rank-$r$ completion, does it always have a nonnegative rank-$r$ completion? Is every rank-$r$ completion of a partial matrix that is nonnegative also a nonnegative rank-$r$ completion of the partial matrix? And finally, how to complete a partial matrix to a nonnegative rank-$r$ matrix? We have summarized the questions and our results studying these questions in~\Cref{table:questions}. 

Analogous questions have been widely studied for rank-$r$ completion of partial matrices. Rank-$1$ completability has been investigated in~\cite{cohen1989ranks,hadwin2006rank,kiraly2013error} and is fully understood. For rank-$2$ completability, Bernstein gives a combinatorial characterization of independent sets in the algebraic matroid associated to rank-2 matrices~\cite{BERNSTEIN20171}. Independent sets correspond to positions of observed entries such that any partial matrix with these observed entries has a rank-2 completion. Singer and Cucuringu use ideas from rigidity theory to study which patterns of observed entries lead to finite and unique rank-$r$ completions~\cite{singer2010uniqueness}. The results mentioned so far use algebraic, combinatorial and geometric tools and hold for specific patterns of observed entries. The results in~\cite{singer2010uniqueness} hold only for generic matrices, and are not true for all matrices. 
In a separate line of research~\cite{candes2012exact,keshavan2008learning,keshavan2010matrix}, bounds on the number of randomly observed entries are given that guarantee the reconstruction of the original matrix together with algorithms for reconstruction. 
The development of matrix completion algorithms constitutes another active area of research~\cite{fazel2002matrix,srebro2003weighted,cai2010singular}.
A survey of low-rank matrix completion can be found in \cite{nguyen2019low}.

Much less is known for low nonnegative rank completion.
Low-rank matrix completion of nonnegative matrices without taking into account nonnegative rank has been studied in \cite{SONG2020106300}. 
Previous work on low nonnegative rank completion has concentrated on algorithms for the completion of a partial matrix to a matrix of low nonnegative rank and their efficiency \cite{xu2012alternating, guglielmi2020efficient, thanh2024minimum}.
A common approach in these algorithms is to combine techniques from both nonnegative factorizations and low-rank matrix completions to approximate the partial matrix with the product of two nonnegative matrices.
The approximation is done by solving a constrained optimization problem with numerical methods such as constrained least squares method \cite{xu2012alternating}, constrained gradient system \cite{guglielmi2020efficient} or MinVol \cite{thanh2024minimum}.
Despite these advances, the study of low nonnegative rank completion remains relatively unexplored.

We explore the geometry of low nonnegative rank matrix completion.
Most of our results are for nonnegative rank-$r$ completion, where $r \in \{1,2,3\}$. The results are summarized in~\Cref{table:questions}. Our first main result is that a partial matrix with nonnegative entries has a nonnegative rank-1 completion if and only if it has a rank-1 completion (\Cref{lem:rank_1_nonnegative_completion_iff_completion}). This is not true in general when $r \geq 2$. As our second main result, for $3 \times 3$ matrices we characterize all the patterns of observed entries when having a rank at most two completion is equivalent to having a nonnegative rank at most two completion (\Cref{thm:3x3_nonnegative_rank_2}). This characterization has somewhat similar flavor to various results in the low rank matrix completion literature~\cite{singer2010uniqueness,BERNSTEIN20171} that characterize when every generic partial matrix with fixed pattern of observed entries is rank-$r$ completable or has finitely many rank-$r$ completions. 

If a partial matrix with nonnegative entries has a rank-$r$ completion that is nonnegative, where $r \in \{1,2\}$, then it has a nonnegative rank-$r$ completion (\Cref{lem:rank_1_nonnegative_completion_iff_completion,lem:nonnegative-rank-2-completion}). This is not true for $r \geq 3$. In the case of $r=3$ and one missing entry, we construct an example by studying the projection map that forgets one entry of the variety of matrices of rank at most $r$ (\Cref{example:one_missing_entry_no_nonnegative_rank_three_completion}). For $r=3$ and two missing entries, we use a geometric characterization for nonnegative rank-$r$ completion employing families of nested polytopes to construct examples where a partial matrix with nonnegative entries has a nonnegative completion of rank three but no nonnegative rank-$3$ completion (\Cref{ex:square,ex:11_22_counter_example}). The geometric characterization for nonnegative rank-$r$ completion generalizes the geometric characterization for nonnegative rank introduced in~\cite{cohen1993nonnegative,vavasis2010complexity}. This geometric characterization is the third main contribution of the paper (\Cref{thm:existence_of_nonnegative_rank-r_completion}).

As it can be seen from~\Cref{table:questions}, we do not study uniqueness and finiteness of completions and completion algorithms for nonnegative rank-$r$ completions where $r \geq 2$. In $r=1$ case, the results follow directly from the results for rank-$1$ completion. Studying the uniqueness and finiteness of nonnegative rank-$r$ completions would be very interesting but more challenging than for rank-$r$ completions since finiteness is not anymore a generic property. The nonnegative rank-$r$ completion algorithms fall outside the focus of the present work and have been studied in~\cite{thanh2024minimum}. We hope that the present paper will motivate further work on this topic.

\begin{table}[ht] 
\begin{tabular}{| c | c | c | c | c |}
\hline
\thead{\#} & \thead{Question} & \thead{$\rank_+ = 1$} & \thead{$\rank_+ = 2$} & \thead{$\rank_+ \geq 3$}\\
\hline
1 & \makecell{Nonnegative  rank-$r$ completability} & \makecell{Prop.~\ref{lem:rank_1_nonnegative_completion_iff_completion}\\ \& Cor.~\ref{cor:nonnegative_rank-1_completion_existence_and_uniqueness}} & \makecell{Prop.~\ref{lem:nonnegative-rank-2-completion} \&\\  Thm.~\ref{thm:3x3_nonnegative_rank_2} ($3 \times 3$)}  & \makecell{Thm.~\ref{thm:existence_of_nonnegative_rank-r_completion}}\\
\hline
2 & \makecell{If there is a rank-$r$ completion, \\ is there also a nonnegative rank-$r$ completion?} & \makecell{Yes: \\ Prop.~\ref{lem:rank_1_nonnegative_completion_iff_completion}} & \makecell{No: \\ Ex.~\ref{ex:rank-2_counter_example_diagonal_missing_entries}} & \makecell{No: \\ Ex.~\ref{example:one_missing_entry_no_nonnegative_rank_three_completion}, \ref{ex:square}, \ref{ex:11_22_counter_example}}\\
\hline
3 & \makecell{Is every rank-$r$ completion that is nonnegative \\ also a nonnegative rank-$r$ completion?} & \makecell{Yes: \\ Prop. \ref{lem:rank_1_nonnegative_completion_iff_completion}} & \makecell{Yes: \\ Prop.~\ref{lem:nonnegative-rank-2-completion}} & \makecell{No: \\ Ex.~\ref{example:one_missing_entry_no_nonnegative_rank_three_completion}, \ref{ex:square}, \ref{ex:11_22_counter_example}}\\
\hline
4 & \makecell{Uniqueness or finiteness of\\ nonnegative rank-$r$
completion} & \makecell{Cor.~\ref{cor:nonnegative_rank-1_completion_existence_and_uniqueness}} & \makecell{} & \makecell{}\\
\hline
5 & \makecell{Completion algorithm} & \makecell{Sec.~\ref{sec:nonnegative_rank_one_completion}} & \makecell{} & \makecell{}\\
\hline
\end{tabular}
\caption{A summary of our main research questions and results.}
\label{table:questions}
\end{table}

The outline of the paper is the following. In~\Cref{sec:background} we recall the geometric description of nonnegative rank (\Cref{sec:geometric_description}) and give some basic results about rank-$r$ matrix completion (\Cref{sec:rank_r_completion}). In~\Cref{sec:nonnegative_rank_one_and_two_completion} we study nonnegative rank-1 (\Cref{sec:nonnegative_rank_one_completion}) and nonnegative rank-2 (\Cref{sec:nonnegative_rank_two_completion}) matrix completion. In~\Cref{sec:one_missing_entry} we consider the case of one missing entry. Finally, in~\Cref{sec:nonnegative_rank_three_completion} we investigate nonnegative rank-$3$ matrix completion: In~\Cref{sec:geometric_description_of_nonnegative_rank_completion} we introduce a geometric characterization of nonnegative rank-$r$ matrix completion. In~\Cref{sec:missing_entries_11_and_21,sec:missing_entries_11_and_22} we study nonnegative rank-$3$ completion of $4 \times 4$ matrices with two missing entries. The code for computations is available at
\url{https://github.com/kaiekubjas/geometry-of-low-nonnegative-rank-matrix-completion}.

\textbf{Notation and conventions.} Unless specified otherwise, we will consider $p \times q$ matrices with real entries. Let $E \subseteq [p] \times [q]$ denote the set of observed entries. We consider partial matrices $M_E \in \mathbb{R}^E$ with the entries in the set $E$ observed. By a nonnegative matrix we mean a matrix with nonnegative entries and by a partial nonnegative matrix we mean a partial matrix with nonnegative entries. Let $M \in \mathbb{R}^{p \times q}$. For index sets $I \subseteq [p]$, $J \subseteq [q]$, we use the following submatrix notation: $M_{I,J}$ denotes the submatrix of $M$ with rows indexed by $I$ and columns indexed by $J$;  $M_{I,:}$ denotes the submatrix of $M$ with rows indexed by $I$ and all columns; $M_{:,J}$ denotes the submatrix of $M$ with all rows and columns indexed by $J$.
For $i \in [p]$, $j \in [q]$, $\hat{i}$ denotes the set $[p]\backslash \{i\}$ and  $\hat{j}$ denotes the set $[q]\backslash \{j\}$. By $\operatorname{cone}(M)$, we denote the conic hull of the columns of the matrix $M$.
We use the notation $\subset$ for proper subsets and $\subseteq$ for any subsets.

\section{Background}\label{sec:background}

In this section, we will recall the geometric description of nonnegative rank (\Cref{sec:geometric_description}) and state some basic results about rank-$r$ matrix completion (\Cref{sec:rank_r_completion}).

\subsection{Geometric description of nonnegative rank} \label{sec:geometric_description}

Let $M \in \mathbb{R}_{\geq 0}^{p \times q}$ be a nonnegative matrix of rank $r$.
The problem of finding a size-$r$ nonnegative factorization of $M$ can be formulated geometrically via nested polyhedral cones in several equivalent ways.
This geometric interpretation is often the key tool with which nonnegative factorizations and their existence is studied.
The geometric interpretation was introduced by Cohen and Rothblum \cite{cohen1993nonnegative}; a widely-used alternative is due to Vavasis \cite{vavasis2010complexity}.
In this paper we use the description due to Vavasis.
Let $M=AB$ be a size-$r$ factorization of $M$.
Here $A,B$ need not be nonnegative matrices.
Define $\hat P = \operatorname{cone}(B)$, which is the conic hull of the columns of $B$, and $\hat Q = \{x \in \mathbb{R}^r \mid Ax \geq 0 \}$.
We have $\hat P \subseteq \hat Q$ since $a_ib_j = m_{ij} \geq 0$ for rows $a_i$ of $A$ and columns $b_j$ of $B$.
The matrix $M$ has a  size-$r$ nonnegative  factorization if and only if there exists a simplicial cone $\hat \Delta$ such that $\hat P \subseteq \hat \Delta \subseteq \hat Q$. 
Indeed, let $c_1, \dots, c_r \in \mathbb{R}^r$ generate the rays of $\hat \Delta$ and let $C \in \mathbb{R}^{r \times r}$ be a matrix with columns $c_1, \dots, c_r \in \mathbb{R}^r$. Then $C$ is invertible since $\hat \Delta$ is a simplicial cone. 
The rows of $C^{-1}$ define the facets of $\hat \Delta$. 
Since $\hat \Delta$ is nested between $\hat P$ and $\hat Q$, we have $AC \geq 0$ and $C^{-1}B \geq 0$.

By taking an affine slice of the nested cones corresponding to the nonnegative matrix $M$, one is able to interpret nonnegative factorizations and nonnegative rank via nested polytopes instead of nested cones.
This allows us to reduce the dimension by 1: instead of considering $r$-dimensional cones, we consider $r-1$-dimensional polytopes $P$ and $Q$, which can be particularly helpful in lower dimensions.
Nonnegative rank can be characterized as follows: a nonnegative matrix of rank $r$ has nonnegative rank $r$ if and only if there exists an $r$-simplex $\Delta$ such that $P \subseteq \Delta \subseteq Q$.
In the rest of the paper, we use $\hat P$ and $\hat Q$ to denote the cones and $P$ and $Q$ to denote the polytopes.

This geometric characterization of nonnegative rank can be used to prove the following theorem from \cite{cohen1993nonnegative}.

\begin{theorem}[{{\cite[Theorem 4.1]{cohen1993nonnegative}}}]\label{lem:nonnegative_rank_and_rank_equal_when_1_2}
    For nonnegative matrices of rank one and two, nonnegative rank is equal to rank.
\end{theorem}

Denote by $\mathcal{V}_r^{p \times q}$ the variety of $p \times q$ matrices of rank $\leq r$ and by $\mathcal{M}_r^{p \times q}$ the set of matrices of nonnegative rank $\leq r$.
\Cref{lem:nonnegative_rank_and_rank_equal_when_1_2} implies that $\mathcal{M}_1^{p \times q} = \mathcal{V}_1^{p \times q} \cap \mathbb{R}^{p \times q}_{\geq 0}$ and $\mathcal{M}_2^{p \times q}=\mathcal{V}_2^{p \times q} \cap \mathbb{R}^{p \times q}_{\geq 0}$. 
If a nonnegative matrix has rank at least three, then its nonnegative rank is at least its rank, but in general the equality does not hold anymore.

In \cite{kubjas2015fixed} the authors provide a geometric algorithm for checking whether a nonnegative matrix has nonnegative rank at most three.
This algorithm uses the characterization of nonnegative factorizations in terms of nested polytopes $P \subseteq Q$ in $\mathbb{R}^2$.

\begin{theorem}[{{\cite[Corollary 4.5]{kubjas2015fixed}}}]\label{cor:rank-3_geometric_algorithm}
A matrix $M \in \mathbb{R}^{p \times q}_{\geq 0}$ has nonnegative rank at most three if and only if 
\begin{enumerate}
    \item $\rank(M) < 3$, or
    \item $\rank(M)=3$, and there exists a triangle $\Delta$ with $P \subseteq \Delta \subseteq Q$ such that a vertex of $\Delta$ coincides with a vertex of $Q$, or 
    \item $\rank(M)=3$, and there exists a triangle $\Delta$ with $P \subseteq \Delta \subseteq Q$ such that an edge of $\Delta$ contains an edge of $P$.
\end{enumerate}
\end{theorem}

For the algorithm it is enough to consider one condition for each vertex of $Q$ and for each edge of $P$.
We use this algorithm extensively in~\Cref{sec:nonnegative_rank_three_completion} when we consider the existence of nonnegative rank-$3$ completions of partial nonnegative matrices.

\subsection{Rank-\texorpdfstring{$r$}{r} completion} \label{sec:rank_r_completion}

We start this subsection with two lemmas on properties of matrix rank that we will use several times throughout the paper. Then in~\Cref{def:rxr_minors_zero_consistency} we introduce $r \times r$ minors zero-consistency. Together with an additional condition on elimination ideals, $r \times r$ minors zero-consistency will form a necessary condition for rank-$r$ matrix completion, as shown in~\Cref{theorem:rank_r_completion}.

\begin{lemma} \label{lemma:removing-one-row-and-one-column}
Let $M \in \mathbb{R}^{p \times q}$ and $i \in [p], j \in [q]$. The following conditions are equivalent:
\begin{enumerate}
\item $\rank(M_{\hat{i},\hat{j}}) \leq r-1$ and $\rank(M) \leq r$;
\item $\rank(M_{\hat{i},\hat{j}}) \leq r-1$, and at least  $\rank(M_{\hat{i},:}) \leq r-1$ or $\rank(M_{:,\hat{j}}) \leq r-1$.
\end{enumerate}
\end{lemma}

\begin{proof}
We start by proving that the first condition implies the second one.  Assume $\rank(M_{\hat{i},\hat{j}}) \leq r-1$ and $\rank(M) \leq r$. Moreover, we assume that $p >r$ and $q >r$, because otherwise the implication is immediate because either $M_{\hat{i},:}$ has at most $r-1$ rows or $M_{:,\hat{j}}$ has at most $r-1$ columns.  If the row $M_{i,:}$ is zero, then $\rank(M_{:,\hat{j}}) = \rank(M_{\hat{i},\hat{j}}) \leq r-1$. Now we assume that the row $M_{i,:}$ is nonzero.
If the row $M_{i,:}$ can be written as a linear combination of the other rows of $M$, then $M_{i,\hat{j}}$ can be written as the same linear combination of rows of $M_{\hat{i},\hat{j}}$. Since $M_{\hat{i},\hat{j}}$ has rank at most $r-1$, then also rank of $M_{:,\hat{j}}$ is at most $r-1$. Finally, if the row $M_{i,:}$ cannot be written as a linear combination of the other rows of $M$, then $\rank(M) > \rank(M_{\hat{i},:})$, which means that $\rank(M_{\hat{i},:}) \leq r-1$. 

The second condition implies the first one, because by adding one row or column to a matrix, the rank can increase by at most one.
\end{proof}

\begin{lemma} \label{lemma:r_minors_zero_consistency}
Let $M \in \mathbb{R}^{p \times q}$ and $I \subset [p], K \subset [q]$. Then $\rank(M) \leq r$ and $\rank(M_{I,K}) \leq r-1$ imply that  $\rank(M_{I,:}) \leq r-1$ or $\rank(M_{:,K}) \leq r-1$.
\end{lemma}

\begin{proof}
Assume $\rank(M) \leq r$ and $\rank(M_{I,K}) \leq r-1$. Let us assume by contradiction that both $M_{I,:}$ and $M_{:,K}$ have rank $r$. This means that there exists $i \in [p]\backslash I$ and $k \in [q] \backslash K$ such that $M_{I \cup \{i\},:}$ and $M_{:,K \cup \{k\}}$ have rank $r$. Then by~\Cref{lemma:removing-one-row-and-one-column}, the matrix $M_{I \cup \{i\},K \cup \{k\}}$ has rank greater than $r$, which is a contradiction to $M$ having rank $r$.
\end{proof}

\Cref{lemma:r_minors_zero_consistency} motivates the following definition.

\begin{definition} \label{def:rxr_minors_zero_consistency}
A partial matrix $ M_E \in \mathbb{C}^E$ is said to be $r \times r$ minors zero-consistent if the following condition holds: If an $r \times r $ submatrix $M_{I,K}$ has rank at most $r-1$, then either the submatrix $M_{I,:}$ or the submatrix $M_{:,K}$ has rank at most $r-1$. 
\end{definition}

Consider the polynomial ring $\mathbb{C}[x_{ij}: 1 \leq i \leq p, 1 \leq j \leq q]$. Let $I_r$ be the ideal of $(r+1) \times (r+1)$ minors of the matrix $(x_{ij})_{1 \leq i \leq p, 1 \leq j \leq q}$. 
For $E \subseteq [p] \times [q]$, let $I_E$ be the elimination ideal $I_r \cap \mathbb{C}[x_{ij}: (i,j) \in E]$.

\begin{proposition} \label{theorem:rank_r_completion}
If a partial matrix $M_E \in \mathbb{C}^E$ has a rank at most $r$ completion $M \in \mathbb{C}^{m \times n}$, then
\begin{enumerate}
\item $M_E \in \mathbb{V}(I_E)$ and
\item the partial matrix $M_E$ is $r \times r$ minors zero-consistent.
\end{enumerate}
\end{proposition}

\begin{proof}[Proof of~\Cref{theorem:rank_r_completion}]
Assume the partial matrix $M_E$ has a rank at most $r$ completion $M$. Then $M \in \mathbb{V}(I_r)$ and thus $M_E \in \mathbb{V}(I_E)$ by~\cite[Chapter 3, \S2, Lemma 1]{cox1997ideals}. The second condition follows from~\Cref{lemma:r_minors_zero_consistency}.
\end{proof}

In the case of rank-1 matrix completion, \Cref{theorem:rank_r_completion} can be made more explicit (\cite{cohen1989ranks} and \cite[Theorem 5]{hadwin2006rank}). We will recall it in~\Cref{lem:zero_row_and_cycle_property_iff_completion}. If $r \geq 2$, a nice description of the generators of the elimination ideal $I_E$ is not known for a general $E$, although in some special cases, it is generated by the $(r+1) \times (r+1)$ minors. An analogue of~\Cref{theorem:rank_r_completion} for rank-1 tensor completion appears in~\cite[Proposition 2.2]{kahle2017geometry}.

The independence of observed entries is often studied in the language of algebraic matroids, see~\cite{rosen2020algebraic} for an introduction. If $E$ is an independent set in the algebraic matroid of the determinantal ideal $I_r$, then $I_E$ is the zero ideal and the condition 1 in~\Cref{theorem:rank_r_completion} is automatically satisfied. In the case of rank-1 completion, the set $E$ is independent if and only if the corresponding bipartite graph does not have cycles (see~\Cref{sec:nonnegative_rank_one_completion} for more details). In the case of rank-2 completion, the independence is characterized in~\cite[Theorem 4.4]{BERNSTEIN20171}.

\begin{question}
Does the other direction of~\Cref{theorem:rank_r_completion} hold, i.e., if conditions 1 and 2 hold, does $M_E$ have a rank-$r$ completion with complex entries? 
\end{question}

\section{Nonnegative rank-1 and nonnegative rank-2 completion} \label{sec:nonnegative_rank_one_and_two_completion}

In this section, we study nonnegative rank-$1$ (\Cref{sec:nonnegative_rank_one_completion}) and nonnegative rank-$2$ (\Cref{sec:nonnegative_rank_two_completion}) matrix completion.

\subsection{Nonnegative rank-1 completion} \label{sec:nonnegative_rank_one_completion}

The main result in this subsection is that a partial matrix with nonnegative entries has a nonnegative rank-1 completion if and only if it has a rank-1 completion (\Cref{lem:rank_1_nonnegative_completion_iff_completion}). Then we recall some results about rank-1 completion in the language of graph theory (\Cref{lem:zero_row_and_cycle_property_iff_completion,lem:rank_1_connected_graph_iff_unique}) and finally state analogous results for nonnegative rank-1 completion (\Cref{cor:nonnegative_rank-1_completion_existence_and_uniqueness}).

\begin{proposition}\label{lem:rank_1_nonnegative_completion_iff_completion}
    A partial matrix with nonnegative entries has a nonnegative rank-$1$ completion if and only if it has a rank-$1$ completion.
\end{proposition}

\begin{proof}
    The forward direction is immediate.
    Conversely, suppose that $M_E \in \mathbb{R}^E_{\geq 0}$ has a rank-$1$ completion $M \in \mathbb{R}^{p \times q}$ that is not necessarily nonnegative.
    Then $M = ab^T$, where $a \in \mathbb{R}^p$ and $b \in \mathbb{R}^q$.
    Taking the entry-wise absolute values $|a|$ and $|b|$ gives a  nonnegative rank-$1$ completion $\tilde M = |a||b|^T$ of $M_E$.
\end{proof}

The rank-$1$ completability of partial matrices is well-studied, see e.g. \cite{cohen1989ranks, hadwin2006rank,kiraly2013error} for completability of general partial matrices, and \cite{kubjas2017matrix} for completions of partial nonnegative matrices to the independence model.
Here we present the main results needed in the nonnegative rank-$1$ case.

A partial matrix $M_E$ has an associated bipartite graph $G(E) = (V_1, V_2, E)$, where $V_1$ are the vertices corresponding to the rows of the matrix $M_E$, $V_2$ are the vertices corresponding to the columns of the matrix $M_E$, and $(i,j) \in E$ if the entry $m_{ij}$ of $M_E$ is observed.  
The partial matrix $M_E$ has the \textit{zero row (column) property}, if given an observed entry that is $0$, all other observed entries in the same row (column) are also zero \cite[Section 2.1]{hadwin2006rank}.
Lastly, let $C = \{ (r_1,c_1), (r_1,c_2), (r_2,c_2), \dots, (r_k, c_k), (r_k, c_1) \}$ be a cycle in the bipartite graph $G$ associated to $M_E$.
Then $M_E$ is called singular with respect to the cycle $C$ if 
\begin{equation}\label{eq:cycle_property}
    m_{r_1c_1}m_{r_2c_2}\cdots m_{r_kc_k} = m_{r_1c_2}m_{r_2c_3}\cdots m_{r_kc_1}.
\end{equation}
In \cite{hadwin2006rank} the equation \eqref{eq:cycle_property} is called the \textit{cycle property}, and $M_E$ is said to have the cycle property, if \eqref{eq:cycle_property} holds for all cycles in the bipartite graph $G$ associated to $M_E$.

The next two lemmas are paraphrased from \cite{hadwin2006rank}. 

\begin{lemma}[{\cite[Lemma 6.2]{cohen1989ranks}, \cite[Theorem 5]{hadwin2006rank}}]\label{lem:zero_row_and_cycle_property_iff_completion}
    A nonzero partial matrix has a rank-$1$ completion if and only if it has the zero row or column property and the cycle property.
\end{lemma}

\Cref{lem:zero_row_and_cycle_property_iff_completion} is the special case of~\Cref{theorem:rank_r_completion} in the rank-1 case: polynomials~\eqref{eq:cycle_property} in the cycle property generate the ideal $I_E$~\cite[Remark 2.5]{kahle2017geometry} and $1 \times 1$ minors zero-consistency is the same as the combination of the zero row or zero column property. 

\begin{lemma}[{\cite[Corollary 6]{hadwin2006rank}}]\label{lem:rank_1_connected_graph_iff_unique}
    Let $M_E$ be a nonzero partial matrix with observed entries $E \subseteq [p] \times [q]$ such that $M_E$ has a rank-$1$ completion $M$.
    Denote by $G(E)^{\#}=(V_1, V_2, E^{\#})$ the subgraph of $G(E)$ with $(i,j) \in E^{\#}$ if $(i,j) \in E$ and $m_{ij} \neq 0$.
    Then $M$ is the unique completion of $M_E$ if and only if $G(E)^{\#}$ is connected.
\end{lemma}

If the completion is unique, there exists an algorithm for the unique completion of the missing entries using the cycle property~\cite[Lemma 6.2]{cohen1989ranks}, \cite[Formula 6]{hadwin2006rank}, \cite[Theorem 34]{kiraly2015algebraic} and \cite[Theorem 2.6]{kiraly2013error}. If a partial matrix $M_E$ satisfies the zero row or column property and the cycle property, but the graph $G(E)^{\#}$ is not connected, there are infinitely many rank-$1$ completions for $M_E$ \cite[Subsection 2.3]{hadwin2006rank}.

Combining the previous results, we get the following corollary about the nonnegative rank-$1$ completions of partial nonnegative matrices, and the number of possible completions.

\begin{corollary}\label{cor:nonnegative_rank-1_completion_existence_and_uniqueness}
    Let $M_E \in \mathbb{R}^E$ be a nonnegative and nonzero partial matrix.
    Then $M_E$ has a nonnegative rank-$1$ completion if and only if $M_E$ has the zero row or column property and the cycle property.
    Furthermore, if $M_E$ satisfies the zero row or column property and the cycle property, then a nonnegative rank-$1$ completion of $M_E$ is the unique nonnegative rank-$1$ completion if and only if the graph $G(E)^{\#}$ is connected.
\end{corollary}

\begin{proof}
    The first claim follows from \Cref{lem:rank_1_nonnegative_completion_iff_completion} and \Cref{lem:zero_row_and_cycle_property_iff_completion}.
    To prove the second claim, suppose that the partial nonnegative matrix $M_E$ has the zero row or column property and the cycle property and consider the graph $G(E)^{\#}$.
    If the graph $G(E)^{\#}$ is connected, then by \Cref{lem:rank_1_connected_graph_iff_unique} $M_E$ has a unique rank-$1$ completion which implies that the nonnegative rank-$1$ completion is also unique.
    On the other hand, if the graph $G(E)^{\#}$ is not connected, then $M_E$ has infinitely many rank-$1$ completions.
    Since there are finitely many rank-$1$ completions that give a given nonnegative rank-$1$ completion by taking absolute values, then it follows that $M_E$ has infinitely many nonnegative rank-$1$ completions.
    This proves the second claim.
\end{proof}

\subsection{Nonnegative rank-2 completion} \label{sec:nonnegative_rank_two_completion}

We start this subsection with~\Cref{lem:nonnegative-rank-2-completion}, which is a consequence of \Cref{lem:nonnegative_rank_and_rank_equal_when_1_2} and which will be our main tool in this subsection.
In the rest of this subsection we consider nonnegative rank at most two completions of partial $3 \times 3$ matrices.
Our main result is~\Cref{thm:3x3_nonnegative_rank_2} which characterizes patterns of  partial nonnegative $3 \times 3$ matrices that are completable to nonnegative rank at most two matrices independently of the observed entries. The rest of the subsection consists of~\Cref{ex:rank-2_counter_example_diagonal_missing_entries} and~\Cref{cor:all_but_one_entry_missing_in_a_row} needed for proving the theorem.

\begin{proposition}\label{lem:nonnegative-rank-2-completion}
    A partial matrix with nonnegative entries has a nonnegative rank-$2$ completion if and only if it has a rank-$2$ completion that is nonnegative. 
\end{proposition}

\begin{theorem}\label{thm:3x3_nonnegative_rank_2}
Let $E \subseteq [3] \times [3]$. Then every partial nonnegative matrix with observed entries in $E$ can be completed to a nonnegative rank at most two matrix in $\mathbb{R}^{3 \times 3}$ if and only if it can be completed to a rank at most two matrix in $\mathbb{R}^{3 \times 3}$, unless all missing entries are in different rows and columns.
\end{theorem}

\begin{example} \label{ex:rank-2_counter_example_diagonal_missing_entries}
 Consider the partial matrix
    \[
    \begin{pmatrix}
    ? & 0 & 1\\
    1 & ? & 0\\
    0 & 1 & ?
    \end{pmatrix}.
    \]
    Its determinant is $1+m_{11} m_{22}m_{33}$. Hence, every rank at most two completion needs to have at least one negative entry. By setting one or two missing entries to one, we also get examples with one and two diagonal missing entries that have rank-$2$ completions but no rank-$2$ completions with nonnegative entries. In the case of three missing diagonal entries, one can show by analyzing the determinant that there needs to be at least two zero entries for the partial matrix not to have a rank-$2$ completion with nonnegative entries.
\end{example}

\begin{lemma}\label{cor:all_but_one_entry_missing_in_a_row}
Let $E \subseteq [p] \times [q]$ be a pattern of observed entries such that 
there exists a row $i \in [p]$ with at most one observed entry $(i,j)$.  
Let $M_E \in \mathbb{R}^E_{\geq 0}$ be a partial matrix with nonnegative 
observed entries.  Assume further that whenever $m_{ij} \neq 0$, the column 
$j$ contains at least one additional nonzero entry.  
If the partial matrix obtained from $M_E$ by deleting row $i$ has a 
nonnegative rank at most $r$ completion, then $M_E$ has a 
nonnegative rank at most $r$ completion. An analogous result holds when there exists a column with at most one observed entry.
\end{lemma}

\begin{proof}
First assume that the matrix without the $i$th row is completed to a matrix of nonnegative rank at most $r$. If there are no observed entries in the $i$th row or the only observed entry is zero, then we set the entries in the $i$th row to be zero. Then the completion has nonnegative rank at most $r$.

Now consider the case when the $i$th row has exactly one nonzero observed entry $m_{ij}$. By assumption there exists  $k \in [p] \backslash \{i\}$ such that $m_{kj} \neq 0$. We take the $i$th row to be $m_{ij}/m_{kj}$ times the $k$th row of $M_E$. This operation does not increase the nonnegative rank. 
\end{proof}

\begin{proof}[Proof of \Cref{thm:3x3_nonnegative_rank_2}]
    If all missing entries are in different rows and columns, then by~\Cref{ex:rank-2_counter_example_diagonal_missing_entries} there are partial matrices that have rank-$2$ but no nonnegative rank-$2$ completions. Assume now that not all missing entries are in different rows and columns. This means that there are at least two missing entries in the same row or column. Let $M_E$ be a partial nonnegative $3 \times 3$ matrix such that $M_E$ has a rank at most two completion.  If there is a row $i$ with at most one observed entry $m_{ij}$ which satisfies that  whenever $m_{ij} \neq 0$, the column 
$j$ contains at least one additional nonzero entry, then by~\Cref{cor:all_but_one_entry_missing_in_a_row} the matrix has a completion of nonnegative rank at most two because the matrix without the $i$th row is a $2 \times 3$ matrix and therefore has a nonnegative rank at most two completion. 

If $m_{ij}\neq 0$, $m_{kj} = 0$ for $k \in [3] \backslash \{i\}$ and $M_E$ can be completed to a matrix of rank at most two, then the submatrix $M_{\hat{i},:}$ must have rank at most one completion. Otherwise, any completion of the missing entries would give a rank-$3$ matrix. We complete $M_{\hat{i},:}$ to a rank-$1$ matrix with nonnegative entries which is possible by~\Cref{lem:rank_1_nonnegative_completion_iff_completion} and the missing entries in the $i$th row with arbitrary nonnegative values. This yields a rank at most two matrix with nonnegative entries, and therefore its nonnegative rank is also at most two. Analogous arguments work if there is a column with at most one observed entry.

The other direction of the statement is immediate.
\end{proof}

\begin{remark}
The analogue of~\Cref{thm:3x3_nonnegative_rank_2} does not hold for larger matrices. It is not possible to complete the partial nonnegative matrix
    \[
    \begin{pmatrix}
    ? & ? & 0 & 1\\
    m_{21} & 1 & 1 & 0\\
    m_{31} & 0 & 1 & 1
    \end{pmatrix}
    \]
    to a nonnegative rank at most two matrix for any $m_{21}, m_{31} \in \mathbb{R}_{\geq 0}$.
    This follows from \Cref{ex:rank-2_counter_example_diagonal_missing_entries} where it was shown that the submatrix $M_{123,234}$ cannot be completed to a nonnegative rank at most two matrix.
    More generally, whenever a partial matrix has a submatrix that cannot be completed to a nonnegative rank at most $r$ matrix, it follows that the partial matrix cannot be completed to a nonnegative rank at most $r$ matrix. 
\end{remark}

\section{One missing entry} \label{sec:one_missing_entry}

In this section, we study the completability of partial matrices with one missing entry.
In the previous section, we showed that for $r \in \{1,2\}$, a partial nonnegative matrix admits a nonnegative rank-$r$ completion if and only if it admits a rank-$r$ completion that is nonnegative.
Our first main result in this section is~\Cref{example:one_missing_entry_no_nonnegative_rank_three_completion}, which demonstrates that this equivalence fails for $r=3$.
Up to this point, we focus on the coordinate projection that forgets one entry of the variety of matrices of rank at most $r$, which will be used to construct~\Cref{example:one_missing_entry_no_nonnegative_rank_three_completion}. 
Our second main result is~\Cref{thm:nonnegative_rank_completability_one_entry_missing}, which characterizes rank-$r$ completability for partial matrices with one missing entry and determines when such a completion is unique.

\begin{definition}
The singular locus of a map $f:X \rightarrow Y$ of varieties consists of the points $p \in X$ where the differential $df:T_pX \rightarrow T_pY$ does not have maximal possible rank.
\end{definition}

\begin{lemma} \label{lemma:singular_locus_one_missing_entry}
Let $(i,j) \in [p] \times [q]$ and $E:= \left( [p] \times [q] \right) \backslash \{(i,j)\}$. The singular locus of the projection $\pi_E: \mathcal{V}^{p \times q}_r \rightarrow \mathbb{R}^E$ is the subvariety of matrices $M \in \mathcal{V}^{p \times q}_r$ where  $M_{\hat{i},:}$ or $M_{:,\hat{j}}$ has rank at most $r-1$. 
\end{lemma}

\begin{proof}
We start by characterizing the differential $d\pi_E$. Consider the matrix multiplication map 
\[
    \mu \colon \mathbb{R}_r^{p \times r} \times \mathbb{R}_r^{r \times q} \to \mathbb{R}_r^{p \times q}, (A,B) \mapsto AB. 
\]
Its Jacobian is 
\renewcommand{\kbldelim}{(}
\renewcommand{\kbrdelim}{)}
\[
\kbordermatrix{
& m_{11} & m_{12} & \cdots & m_{1q} & m_{21} & m_{22} & \cdots & m_{2q} & \cdots & m_{p1} & m_{p2} & \cdots & m_{pq} \\
a_{11} & b_{11} & b_{12} & \cdots & b_{1q} & 0 & 0 & \cdots & 0 & \cdots & 0 & 0 & \cdots & 0 \\
a_{12} & b_{21} & b_{22} & \cdots & b_{2q} & 0 & 0 & \cdots & 0 & \cdots & 0 & 0 & \cdots & 0 \\
\vdots & \vdots & \vdots & \ddots & \vdots & \vdots & \vdots & \ddots & \vdots & \vdots & \vdots & \vdots & \ddots & \vdots \\
a_{1r} & b_{r1} & b_{r2} & \cdots & b_{rq} & 0 & 0 & \cdots & 0 & \cdots & 0 & 0 & \cdots & 0 \\
a_{21} & 0 & 0 & \cdots & 0 & b_{11} & b_{12} & \cdots & b_{1q} & \cdots & 0 & 0 & \cdots & 0 \\
a_{22} & 0 & 0 & \cdots & 0 & b_{21} & b_{22} & \cdots & b_{2q} & \cdots & 0 & 0 & \cdots & 0 \\
\vdots & \vdots & \vdots & \ddots & \vdots & \vdots & \vdots & \ddots & \vdots & \vdots & \vdots & \vdots & \ddots & \vdots \\
a_{2r} &0 & 0 & \cdots & 0 & b_{r1} & b_{r2} & \cdots & b_{rq} & \cdots & 0 & 0 & \cdots & 0\\
\vdots & \vdots & \vdots & \vdots & \vdots & \vdots & \vdots & \vdots & \vdots & \vdots & \vdots & \vdots & \vdots \\
a_{p1} & 0 & 0 & \cdots & 0 & 0 & 0 & \cdots & 0 & \cdots & b_{11} & b_{12} & \cdots & b_{1q} \\
a_{p2} & 0 & 0 & \cdots & 0 & 0 & 0 & \cdots & 0 & \cdots & b_{21} & b_{22} & \cdots & b_{2q}\\
\vdots & \vdots & \vdots & \ddots & \vdots & \vdots & \vdots & \ddots & \vdots & \vdots & \vdots & \vdots & \ddots & \vdots \\
a_{pr} & 0 & 0 & \cdots & 0 & 0 & 0 & \cdots & 0 & \cdots & b_{r1} & b_{r2} & \cdots & b_{rq}\\
b_{11} & a_{11} & 0 & \cdots & 0 & a_{21} & 0 & \cdots & 0 & \cdots & a_{p1} & 0 & \cdots & 0 \\
b_{12} & 0 & a_{11} & \cdots & 0 & 0 & a_{21} & \cdots & 0 & \cdots & 0 & a_{p1} & \cdots & 0 \\
\vdots & \vdots & \vdots & \ddots & \vdots & \vdots & \vdots & \ddots & \vdots & \vdots & \vdots & \vdots & \ddots & \vdots \\
b_{1q} & 0 & 0 & \cdots & a_{11} & 0 & 0 & \cdots & a_{21} & \cdots & 0 & 0 & \cdots & a_{p1} \\
b_{21} & a_{12} & 0 & \cdots & 0 & a_{22} & 0 & \cdots & 0 & \cdots & a_{p2} & 0 & \cdots & 0 \\
b_{22} & 0 & a_{12} & \cdots & 0 & 0 & a_{22} & \cdots & 0 & \cdots & 0 & a_{p2} & \cdots & 0 \\
\vdots & \vdots & \vdots & \ddots & \vdots & \vdots & \vdots & \ddots & \vdots & \vdots & \vdots & \vdots & \ddots & \vdots \\
b_{2q} & 0 & 0 & \cdots & a_{12} & 0 & 0 & \cdots & a_{22} & \cdots & 0 & 0 & \cdots & a_{p2} \\
\vdots & \vdots & \vdots & \vdots & \vdots & \vdots & \vdots & \vdots & \vdots & \vdots & \vdots & \vdots & \vdots \\
b_{r1} & a_{1r} & 0 & \cdots & 0 & a_{2r} & 0 & \cdots & 0 & \cdots & a_{pr} & 0 & \cdots & 0 \\
b_{r2} & 0 & a_{1r} & \cdots & 0 & 0 & a_{2r} & \cdots & 0 & \cdots & 0 & a_{pr} & \cdots & 0 \\
\vdots & \vdots & \vdots & \ddots & \vdots & \vdots & \vdots & \ddots & \vdots & \vdots & \vdots & \vdots & \ddots & \vdots \\
b_{rq} & 0 & 0 & \cdots & a_{1r} & 0 & 0 & \cdots & a_{2r} & \cdots & 0 & 0 & \cdots & a_{pr}
}
\]
The tangent space $T_M \mathcal{V}^{p \times q}_r$ at the matrix $M=AB$ is spanned by the rows of the above Jacobian. The differential $d\pi_E$ maps it to the row span of the above Jacobian with the column corresponding to the entry $m_{ij}$ removed. Therefore, the singular locus of $\pi_E$ consists of $M \in \mathcal{V}^{p \times q}_r$ such that Jacobian with the column corresponding to $m_{ij}$ removed does not have maximal possible rank.

Without loss of generality, we assume that the missing entry is $(1,1)$ and from now on we consider the above Jacobian with the first column removed. It has size $(pr+qr) \times (pq-1)$ and its rank at a generic point is $pr+qr-r^2$. A matrix $M$ is in the singular locus of $\pi_E$ if the rank of the corresponding Jacobian is less than $pr+qr-r^2$.

First assume that the rank of $M_{\hat{1},:}$ is at most $r-1$. In this case, the rank of the submatrix of the Jacobian with columns indexed by the entries of the submatrix  $M_{\hat{1},:}$ is at most $(p-1)(r-1)+q(r-1)-(r-1)^2$. There are $q-1$ columns of the Jacobian corresponding to the entries of the first row of $M_{\hat{1},:}$ without the $(1,1)$-entry. The corresponding columns can increase the rank of the Jacobian at most by $q-1$. Therefore, the rank of the Jacobian is at most 
\[
(p-1)(r-1)+q(r-1)-(r-1)^2 + (q-1)=pr+qr-r^2-1+(r-p) \leq pr+qr-r^2-1
\]
because $p \geq r$. Therefore, $M$ is in the singular locus of $\pi_E$. An analogous argument works if the rank of $M_{:,\hat{1}}$ is at most $r-1$. 

For the converse direction that a matrix in the singular locus implies that $M_{\hat{1},:}$ or $M_{:,\hat{1}}$ has rank at most $r-1$, we consider two different cases. First assume that either the first row of $M$ is not a linear combination of the other rows of $M$ or the first column of $M$ is not a linear combination of the other columns of $M$. Then either the rank of $M_{\hat{1},:}$ is at most $r-1$, and in the second case the rank of $M_{:,\hat{1}}$ is at most $r-1$.
 
For the second case assume that both the first row of $M$ is a linear combination of the other rows of $M$ and the first column of $M$ is a linear combination of the other columns of $M$. Consider the submatrix of $M$ without the first row. Since it has rank at most $r$, it has at most $r$ linearly independent rows. Without loss of generality, we can assume that each of the rows $r+2,\ldots,p$ can be written as a linear combination of the rows $2,\ldots,r+1$. Similarly, without loss of generality, we can assume that each of the columns $r+2,\ldots,q$ can be written as a linear combination of the columns $2,\ldots,r+1$. Let $S:=\{1,r+2,\ldots,p\}$ and $T:=\{1,r+2,\ldots,q\}$. 
The entries indexed by the elements in $(S \times T) \backslash \{(1,1)\}$ can be written in terms of the other entries of the partial matrix $M$ by our assumptions.
The same is true for the corresponding columns of the Jacobian. 
Therefore, we consider from now on the submatrix $J'$ of the Jacobian with these columns removed, which has precisely $pr+qr-r^2$ columns. 
A point belongs to the singular locus of the projection $\pi_E$ if and only if the rank of the Jacobian $J'$ at this point is less than $pr+qr-r^2$ which means that there exists a column of the Jacobian $J'$ that can be written as a nontrivial linear combination of other columns of the Jacobian. 

Assume that there exist coefficients $\lambda_{ij} \in \mathbb{R}$, not all of them zero, such that the linear combination of the columns of the Jacobian $J'$ with these coefficients add to the zero column. If there exists $i \in S$ such that at least one of $\lambda_{i,2},...,\lambda_{i,r+1}$ is nonzero, then by looking at the rows indexed by $a_{i,1},\ldots,a_{i,r}$ of the Jacobian, we get that the second to $r+1$-st column of $B$ are linearly dependent, i.e., the rank of this matrix is at most $r-1$. By earlier assumptions, the matrix that is obtained from $B$ by deleting the first column has rank at most $r-1$ as well and hence for the matrix that is obtained from $M$ by deleting the first column has rank at most $r-1$.

If there exists $j \in T$ such that at least one of $\lambda_{2,j},...,\lambda_{r+1,j}$ is nonzero, then by looking at the rows indexed by $b_{1,j},\ldots,b_{r,j}$ of the Jacobian, we get that the second to $r+1$-st row of $A$ are linearly dependent, i.e., the rank of this matrix is at most $r-1$. By earlier assumptions, the matrix that is obtained from $A$ by deleting the first row has rank at most $r-1$ as well and hence for the matrix that is obtained from $M$ by deleting the first row has rank at most $r-1$.

Assume now that we are not in one of the previous cases. Let $k,l \in \{2,\ldots,r+1\}$ be such that $\lambda_{k,l} \neq 0$. Since we have excluded the previous cases, we have that $\lambda_{k,1}=\lambda_{k,r+2}= \ldots = \lambda_{k,q}=0$. Then by looking at the rows indexed by $a_{k,1},\ldots,a_{k,r}$ of the Jacobian, we get that the second to $r+1$-st column of $B$ are linearly dependent, i.e., the rank of this matrix is at most $r-1$. By earlier assumptions, the matrix that is obtained from $B$ by deleting the first column has rank at most $r-1$ as well and hence for the matrix that is obtained from $M$ by deleting the first column has rank at most $r-1$.
\end{proof}

\begin{remark}
The Jacobian appearing in the proof of~\Cref{lemma:singular_locus_one_missing_entry} is the transpose of the completion matrix in~\cite{singer2010uniqueness}. 
\end{remark}

\begin{proposition}\label{prop:one_entry_missing_homeomorphic}
Let $(i,j) \in [p] \times [q]$ and $E:= \left( [p] \times [q] \right) \backslash \{(i,j)\}$. Then $\pi_E: \mathcal{V}^{p \times q}_r \rightarrow \mathbb{R}^E$ is a homeomorphism on  $\mathcal{M}^{p \times q}_r$ outside the singular locus of $\pi_E$. Outside the singular locus of $\pi_E$, the points of the boundary of $\mathcal{M}^{p \times q}_r$ map to the boundary of $\pi_E(\mathcal{M}^{p \times q}_r)$ and the points of the interior of $\mathcal{M}^{p \times q}_r$ map to the interior of $\pi_E(\mathcal{M}^{p \times q}_r)$. The points on the singular locus map to the boundary of $\pi_E(\mathcal{M}^{p \times q}_r)$.
\end{proposition}

\begin{proof}
By~\Cref{lemma:singular_locus_one_missing_entry}, the singular locus of $\pi_E$ consists of points where $M_{\hat{i},:}$ or $M_{:,\hat{j}}$ is of rank at most $r-1$. By~\Cref{lemma:removing-one-row-and-one-column}, partial matrices in the image of $\pi_E$ cannot have $\rank(M_{\hat{i},\hat{j}}) \leq r-1$ but $\rank(M_{\hat{i},:})=\rank(M_{:,\hat{j}})=r$. 
Therefore the points of $\mathcal{V}^{p \times q}_r$ outside the singular locus of $\pi_E$ have an $r \times r$ minor of $M_{\hat{i},\hat{j}}$ that is nonzero. 
Let $K \subseteq [p]\backslash \{i\}$ and $L \subseteq [q] \backslash \{j\}$ satisfying $|K|=|L|=r$ be such that $\det (M_{K,L}) \neq 0$. 
Consider the minor $\det( M_{K \cup \{i\},L \cup \{j\}})$. 
Using Laplace expansion, we can write the determinant $\det(M)=m_{ij} \det (M_{K,L})-f$, where $f$ does not depend on the entry $m_{ij}$. 
The inverse of the projection $\pi_E$ on the open set defined by $\det (M_{K,L}) \neq 0$ is given by $m_{ij}=\frac{f}{\det (M_{K,L})}$ for the missing entry $m_{ij}$. 
The set $\pi_E(\mathcal{V}^{p \times q}_r)$ can be covered by open sets such that $\det (M_{K,L}) \neq 0$ holds on the open set for some  $K \subseteq [p]\backslash \{i\}$ and $L \subseteq [q] \backslash \{j\}$ satisfying $|K|=|L|=r$. 
Moreover, the inverse maps for different sets $K,L$ agree on the intersection of the open sets. 
The map $\pi_E$ is a bijection on $\mathcal{V}^{p \times q}_r$ outside the singular locus of $\pi_E$ with the inverse defined on each open set as above and both functions are continuous. 
Therefore, $\pi_E$ is a homeomorphism both on $\mathcal{V}^{p \times q}_r$ and $\mathcal{M}^{p \times q}_r$ outside the singular locus of $\pi_E$.

The statements about interior and boundary outside singular locus follow by general properties of homeomorphisms. The image $M_E$ of a matrix in the singular locus of $\pi_E$ has in every neighborhood partial matrices that cannot be completed to a rank-$r$ matrix. This can be seen by modifying $M_E$ by an arbitrarily small modification  so  that $M_{\hat{i},:}$ and $M_{:,\hat{j}}$ have rank exactly $r$ and $M_{\hat{i},\hat{j}}$ has rank $r-1$.
\end{proof}

We illustrate this result on~\Cref{example:non_trivial_boundary_4x4}. For this, we need the following lemma.

\begin{lemma} \label{lemma:replace_zero_by_negative_entry}
Let $M \in \mathbb{R}^{p \times q}$ be a matrix of rank $r$ and nonnegative rank $r$. Let $(A,B)$ be a unique size-$r$ nonnegative matrix factorization of $M$ up to scalings and permutations. Let $\tilde{A}$ be a matrix that is obtained from $A$ by changing one zero entry to a negative number $\varepsilon$. Then there exists $\varepsilon_0<0$ such that for $\varepsilon \in (\varepsilon_0,0)$, the matrix $\tilde{M}:=\tilde{A}B$ does not have a size-$r$ nonnegative factorization.
\end{lemma}

\begin{proof}
All size-$r$ factorizations of $M$ are of the form $(AC,C^{-1}B)$, where $C \in \operatorname{GL}_r(\mathbb{R})$.  Since $(A,B)$ is a unique size-$r$ nonnegative matrix factorization of $M$, either $AC$ or $C^{-1}B$ contains a negative entry unless $C$ is a composition of a positive scaling and a permutation matrix. 

Suppose $A$ and $\tilde{A}$ differ by the entry $(i,j)$. We will show that if either factor of $(AC,C^{-1}B)$ contains a negative entry, then also a factor of $(\tilde{A}C,C^{-1}B)$ contains a negative entry. The factorization $(\tilde{A}C,C^{-1}B)$ differs from $(AC,C^{-1}B)$ only in the $i$th row of $\tilde{A}C$. Therefore, if $C^{-1}B$ contains a negative entry or $AC$ contains a negative entry outside the $i$th row, then $(\tilde{A}C,C^{-1}B)$ contains a negative entry as well. The elements in the $i$th row of $\tilde{A}C$ satisfy $(\tilde{A}C)_{ik}=(AC)_{ik}+\varepsilon C_{jk}$. Hence if $(AC)_{ik}<0$, then there exists $\varepsilon_0 <0$ such that for $\varepsilon \in (\varepsilon_0,0)$ we have $(\tilde{A}C)_{ik}<0$.
\end{proof}

\begin{example} \label{example:non_trivial_boundary_4x4}
In this example, we will show for $4 \times 4$ matrices of rank three and nonnegative rank three that if we project to partial matrices missing one entry, then the boundary of this set contains matrices from the projection of the nontrivial boundary of $\mathcal{M}^{4 \times 4}_3$. 
In other words, the set $\pi_E(\mathcal{M}_3^{4 \times 4})$ has nontrivial boundaries which was not the case with  nonnegative rank-$1$ and -$2$.

Consider the matrices
\[
M = 
\begin{pmatrix}
12& 2& 1& 9\\ 
8& 6& 3& 10\\ 
4& 16& 13& 9\\ 
12& 4& 6& 5
\end{pmatrix}, \quad
A = 
\begin{pmatrix}
0 & 1 & 2\\ 
1 & 0 & 2\\ 
4 & 1 & 0\\ 
1 & 3 & 0
\end{pmatrix}, \quad
B = 
\begin{pmatrix}
0 & 4 & 3 & 2\\ 
4 & 0 & 1 & 1\\ 
4 & 1 & 0 & 4
\end{pmatrix}.
\]
Then $(A,B)$ is the unique size-$3$ nonnegative matrix factorization of $M$ up to permutations and scalings \cite[Section 6]{kubjas2015fixed}. 

Let $E = \left([4] \times [4] \right) \backslash \{(1,1)\}$. The partial matrix $M_E$ has a unique rank-$3$ completion, which is $M$. It is also the unique nonnegative rank-$3$ completion of $M_E$.

We claim that every neighborhood of $M_E$ has partial matrices that have a nonnegative rank-$3$ completion and points that have a rank-$3$ completion but no nonnegative rank-$3$ completion. The former is easy to show by replacing a zero entry in $A$ and $B$ by a small positive $\varepsilon$ and by considering matrices obtained as their product.   

To show the latter, let $A(\varepsilon)$ be the matrix that is obtained from $A$ by replacing the $(1,1)$ entry in $A$ by negative $\varepsilon$ and let $M(\varepsilon):=A(\varepsilon) B$. By~\Cref{lemma:replace_zero_by_negative_entry},  there exists $\varepsilon_0<0$ such that the matrix $M(\varepsilon)$ does not have a size-3 nonnegative factorization for $\varepsilon \in (\varepsilon_0,0)$. By taking $\varepsilon$ sufficiently close to zero, we can construct $M(\varepsilon)_E$ arbitrarily close to $M_E$. Since $M(\varepsilon)_E$ has a unique rank-$3$ completion, then it must be equal to $M(\varepsilon)$, which does not have nonnegative rank three. Hence $M(\varepsilon)_E$ has a rank-$3$ but not a nonnegative rank-$3$ completion. 

This example generalizes to matrices of size $p \times q$ of rank three and nonnegative rank three by adding $p-4$ nonnegative rows to $A$ and $q-4$ nonnegative columns to $B$.
\end{example}

\begin{example} \label{example:one_missing_entry_no_nonnegative_rank_three_completion}
    In \Cref{example:non_trivial_boundary_4x4} we show that in every neighborhood of the partial nonnegative matrix $M_E$ there exists a partial matrix $M(\varepsilon)_E$ with no nonnegative rank-$3$ completions.
    One example of such a partial matrix is 
    \begin{equation*}
        M(\varepsilon)_E = \begin{pmatrix}
            ? & 2 & 1 & 9 \\
            8 & 6 & 3 & 10 \\
            4 & 16 & 13 & 9 \\
            8 & 3 & 6 & 1
        \end{pmatrix}.
    \end{equation*}
    This matrix has a unique rank-$3$ completion 
    \begin{equation*}
        \begin{split}
            M(\varepsilon)&= \begin{pmatrix}
            12 & 2 & 1 & 9 \\
            8 & 6 & 3 & 10 \\
            4 & 16 & 13 & 9 \\
            8 & 3 & 6 & 1
        \end{pmatrix} = 
        \begin{pmatrix}
            0 & 1 & 2\\ 
            1 & 0 & 2\\ 
            4 & 1 & 0\\ 
            1 & 3 & -1
        \end{pmatrix}
        \begin{pmatrix}
            0 & 4 & 3 & 2\\ 
            4 & 0 & 1 & 1\\ 
            4 & 1 & 0 & 4
        \end{pmatrix} \\&= \underbrace{\begin{pmatrix}
            0 & 1 & 2\\
            1/5 & -4/5 & 3/5\\
            4/5 & -11/5 & -28/5\\
            1/5 & 11/5 & -12/5
        \end{pmatrix}}_{\hat A} \underbrace{\begin{pmatrix}
            44 & 27 & 19 & 42\\
            4 & 0 & 1 & 1\\
            4& 1 & 0&4
        \end{pmatrix}}_{\hat B}.
            \end{split}
    \end{equation*}
    The polytopes $P \subseteq Q$ associated to $M(\varepsilon)$ are depicted in \Cref{fig:one_missing_entry_counter_example}. 
    The polytopes were constructed from the cones $\{(x,y,z) \in \mathbb{R}^3 \mid \hat A x \geq 0 \}$ and $\operatorname{cone}(\hat B)$ by taking the slice $x=1$. 
    We have shown computationally using the \Cref{cor:rank-3_geometric_algorithm} that there does not exist a triangle $\Delta$ nested between $P$ and $Q$.
    The code for this computation can be found in the git repository of the project.
    \begin{figure}[htbp!]
        \centering
        \includegraphics[width=0.45\linewidth]{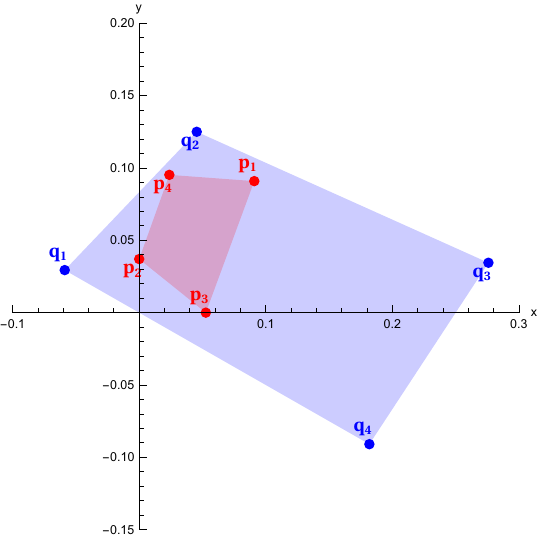}
        \caption{The nested polytopes associated to $M(\varepsilon)$ in \Cref{example:one_missing_entry_no_nonnegative_rank_three_completion}.}
        \label{fig:one_missing_entry_counter_example}
    \end{figure}
\end{example}

\begin{example}
It was shown in~\cite[Example 6.2]{kubjas2015fixed} that the algebraic boundary of $\mathcal{M}^{4 \times 4}_3$ contains $16$ irreducible components $\{m_{ij} = 0 \}$ corresponding to zero entries, and $288$ irreducible components corresponding to different zero patterns of factorizations such as 
\begin{equation}\label{eq:boundary_components_zero_patterns}
    \begin{pmatrix}
            0 & a_{12} & a_{13}\\
            a_{21} & 0 & a_{23}\\
            a_{31} & a_{32} & 0\\
            a_{41} & a_{42} & 0
        \end{pmatrix} \begin{pmatrix}
            0 & b_{12} & b_{13} & b_{14}\\
            b_{21} & 0 & b_{23} & b_{24}\\
            b_{31} & b_{32} & 0 & b_{34}
        \end{pmatrix}.
\end{equation}
The latter algebraic boundary components are generated by maximal minors of a $4 \times 5$ matrix. The algebraic boundary components of $\mathcal{M}^{4 \times 4}_3$ can give algebraic boundary components of the semialgebraic set $\pi_E(\mathcal{M}^{4 \times 4}_3)$. For example, when $E= \left( [4] \times [4] \right) \backslash \{(1,1)\}$, then the zero pattern~\eqref{eq:boundary_components_zero_patterns}
gives an algebraic boundary component of $\pi_E(\mathcal{M}^{4 \times 4}_3)$ that is generated by the polynomial: 
\begin{equation*}
\begin{split}
&m_{14}m_{23}m_{32}m_{33}m_{41}m_{42}-m_{14}m_{22}m_{33}^2m_{41}m_{42}+
m_{12}m_{24}m_{33}^2m_{41}m_{42}-\\
&m_{13}m_{23}m_{32}m_{34}m_{41}m_{42}+m_{13}m_{22}m_{33}m_{34}m_{41}m_{42}-m_{12}m_{23}m_{33}m_{34}m_{41}m_{42}-\\
&m_{14}m_{23}m_{31}m_{33}m_{42}^2+m_{13}m_{23}m_{31}m_{34}m_{42}^2-m_{14}m_{23}m_{32}^2m_{41}m_{43}+\\
&m_{14}m_{22}m_{32}m_{33}m_{41}m_{43}-m_{12}m_{24}m_{32}m_{33}m_{41}m_{43}+m_{12}m_{23}m_{32}m_{34}m_{41}m_{43}+\\
&m_{14}m_{23}m_{31}m_{32}m_{42}m_{43}+m_{14}m_{22}m_{31}m_{33}m_{42}m_{43}-m_{12}m_{24}m_{31}m_{33}m_{42}m_{43}-\\
&m_{13}m_{22}m_{31}m_{34}m_{42}m_{43}-m_{14}m_{22}m_{31}m_{32}m_{43}^2+m_{12}m_{24}m_{31}m_{32}m_{43}^2+\\
&m_{13}m_{23}m_{32}^2m_{41}m_{44}-m_{13}m_{22}m_{32}m_{33}m_{41}m_{44}-m_{13}m_{23}m_{31}m_{32}m_{42}m_{44}+\\
&m_{12}m_{23}m_{31}m_{33}m_{42}m_{44}+m_{13}m_{22}m_{31}m_{32}m_{43}m_{44}-m_{12}m_{23}m_{31}m_{32}m_{43}m_{44}.
\end{split}
\end{equation*}
This polynomial is of degree six and has $24$ terms.
It vanishes on the matrix in~\Cref{example:non_trivial_boundary_4x4}.

Other zero patterns and projections that forget one entry can give boundary components of different degrees and number of terms. They can be easily computed using elimination in~\texttt{Macaulay2}. By contrast, it is not clear how to find a semialgebraic description of $\pi_E(\mathcal{M}^{4 \times 4}_3)$.
\end{example}

\begin{theorem}\label{thm:nonnegative_rank_completability_one_entry_missing}
Let $(i,j) \in [p] \times [q]$ and $E:= \left( [p] \times [q] \right) \backslash \{(i,j)\}$. 

\begin{enumerate}
\item A partial matrix $M_E \in \mathbb{R}^{E}$ is completable to a  matrix in $\mathbb{R}^{p \times q}$ of rank at most $r$ if and only if one of the following conditions holds:
\begin{enumerate}
\item $M_{\hat{i},:}$ or $M_{:,\hat{j}}$ has rank at most $r-1$;
\item $M_{\hat{i},:}$, $M_{:,\hat{j}}$, and $M_{\hat{i},\hat{j}}$ have rank exactly $r$.
\end{enumerate}
The condition (a) corresponds to matrices in the image of the singular locus of the projection map. 

\item A partial matrix $M_E \in \mathbb{R}^{E}$ has infinitely many completions in $\mathbb{R}^{p \times q}$ of rank at most $r$ if and only if $M_{\hat{i},:}$ or $M_{:,\hat{j}}$ has rank at most $r-1$ (which is the same as condition 1(a), and also the same as $M_E$ being in the image of the singular locus). Moreover, if there are infinitely many completions, then the missing entry can be completed to any real number.

\item A partial matrix $M_E \in \mathbb{R}^{E}$ has a unique completion in $\mathbb{R}^{p \times q}$ of rank at most $r$ if and only if $M_{\hat{i},:}$, $M_{:,\hat{j}}$, and $M_{\hat{i},\hat{j}}$ have rank exactly $r$ (which is the same as condition 1(b)). This is the same as the existence of an $r \times r$ minor of $M_{\hat{i},\hat{j}}$ that is nonzero.

\item If $r\geq 3$, to check whether a partial matrix $M_E \in \mathbb{R}^{E}$ has a completion of nonnegative rank at most $r$ cannot be determined by checking whether the matrix has a completion of rank at most $r$ with nonnegative entries. 
\end{enumerate}
\end{theorem}

\begin{proof}
\begin{enumerate}
    \item First assume that $M_E \in \mathbb{R}^{E}$ is completable to a  matrix of rank at most $r$. If neither $M_{\hat{i},:}$ nor $M_{:,\hat{j}}$ has rank at most $r-1$, then $M_{\hat{i},:}$ and $M_{:,\hat{j}}$ have rank exactly $r$. If $M_{\hat{i},\hat{j}}$ would have rank at most $r-1$, then by~\Cref{lemma:removing-one-row-and-one-column}, one of $M_{\hat{i},:}$ and $M_{:,\hat{j}}$ needs to have rank at most $r-1$ as well. Since this is not possible, then $M_{\hat{i},\hat{j}}$ has rank exactly $r$.

    Conversely, if $M_{\hat{i},:}$ or $M_{:,\hat{j}}$ has rank at most $r-1$, then adding any row or column increases rank by at most one, so $M_E$ is completable to a matrix of rank at most $r$. If $M_{\hat{i},:}$, $M_{:,\hat{j}}$, and $M_{\hat{i},\hat{j}}$ have rank exactly $r$, then we can write $M_{i,\hat{j}}$ as a linear combination of the rows of $M_{\hat{i},\hat{j}}$. If we take $M_{i,j}$ to be equal to the same linear combination of the elements of $M_{\hat{i},:}$, then $M_{i,:}$ is a linear combination of the rows of $M_{\hat{i},:}$ and thus the rank of $M$ is $r$.

    \item If $M_{\hat{i},:}$ has rank at most $r-1$, then regardless of what the $ij$th entry is completed, the completion can have rank at most $r$. Therefore, in this case there are infinitely many real and infinitely many nonnegative completions.  
    
    On the other hand, if $M_{\hat{i},:}$, $M_{:,\hat{j}}$, and $M_{\hat{i},\hat{j}}$ have rank exactly $r$, then there is an $r \times r$ minor of $M_{\hat{i},\hat{j}}$ that is nonzero. 
    Let $K \subseteq [p]\backslash \{i\}$ and $L \subseteq [q] \backslash \{j\}$ with $|K|=|L|=r$ such that $\det M_{K,L} \neq 0$. 
    Consider the minor $\det (M_{K \cup \{i\},L \cup \{j\}})$. 
    Using Laplace expansion, we can write the determinant $\det(M)=m_{ij}\det (M_{K,L})-f$, where $f$ does not depend on the entry $m_{ij}$. 
    Then the missing entry needs to be equal to $m_{ij}=\frac{f}{\det M_{K,L}}$, and hence there is a unique completion.

    \item This follows from the previous item. 

    \item This follows from \Cref{example:non_trivial_boundary_4x4}.
\end{enumerate}
\end{proof}

The unique completion in the item (3) is nonnegative if and only if the expression $m_{ij}=\frac{f}{\det M_{K,L}}$ in the item (2) is nonnegative. As having a nonnegative rank-2 completion is the same as having a rank-2 completion that is nonnegative, then in this case, one can easily check whether a partial matrix with one missing entry has a nonnegative rank-2 completion.

\section{Nonnegative rank-3 completion} \label{sec:nonnegative_rank_three_completion}

In this section, we give a geometric characterization of nonnegative rank-$r$ matrix completion (\Cref{sec:geometric_description_of_nonnegative_rank_completion}). We also study nonnegative rank-$3$ matrix completion with two missing entries in one column or row (\Cref{sec:missing_entries_11_and_21}) or in different columns and rows (\Cref{sec:missing_entries_11_and_22}).

\subsection{Geometric description of nonnegative rank completion} \label{sec:geometric_description_of_nonnegative_rank_completion}

In this subsection, we extend the geometric characterization of nonnegative rank in~\Cref{sec:background} to a geometric characterization of nonnegative rank-$r$ completions. The main result is stated in~\Cref{thm:existence_of_nonnegative_rank-r_completion}. 

In~\Cref{sec:background}, we associated a pair of nested polytopes to a nonnegative matrix. The matrix has nonnegative rank $r$ if there exists an $(r-1)$-simplex nested between the polytopes. To a partial matrix with nonnegative entries, we can associate a family of nested polytopes. The matrix has a nonnegative rank-$r$ completion if, for one pair of nested polytopes in this family, there exists an $(r-1)$-simplex that can be nested between the polytopes. 

Let $M_E \in \mathbb{R}^{E}_{\geq 0}$ be a partial nonnegative matrix. We will construct matrices $A \in \mathbb{R}^{p \times r}$ and $B \in \mathbb{R}^{r \times q}$ such that $M=AB$ is a rank-$r$ completion of $M_E$. We will use information from the partial matrix $M_E$ to determine the entries of $A,B$. We do not require $A$ and $B$ to have nonnegative entries. Let $M_{I,J}$ be a maximal observed submatrix of $M_E$ for some $I \subseteq [p]$ and $J \subseteq [q]$. If $\rank(M_{I,J}) = r$, then we can take $A_I$ and $B_J$ such that they give a rank-$r$ factorization of $M_{I,J}$. In particular, this fixes facets of $Q$ corresponding to the rows in $I$ and vertices of $P$ corresponding to the columns in $J$. The remaining entries of $A$ and $B$ are not fully specified but have restrictions imposed by the observed entries of $M_E$ not in $M_{I,J}$. Consequently, the  vertices of $P$ and facets of $Q$ can move as well, but there are restrictions on their movement imposed by the observed entries of $M_E$ not in $M_{I,J}$.
The configurations for the polytopes $P$ and $Q$ such that $P \subseteq Q$ correspond to the nonnegative completions of $M_E$.
This is because whenever $P \subseteq Q$ the completion of $M_E$ is the slack matrix $S_{P,Q}$ of the pair $(P,Q)$ which is nonnegative.

In some contexts, it is easier to work with nested polytopes, and in other contexts, with nested cones that are obtained from the polytopes by taking the conic hull. We will use both descriptions depending on the context. 
In the following example we demonstrate the geometric construction using both nested cones and polytopes.

\begin{example} \label{example:11_and_21_missing}
    Let $E= [4] \times [4] \setminus \{(1,1),(2,1)\}$ and consider the rank-$3$ completions of the partial nonnegative matrix 
    \begin{equation*}
         M_E = \begin{pmatrix}
            ? & m_{12} & m_{13} & m_{14} \\
            ? & m_{22} & m_{23} & m_{24} \\
            m_{31} & m_{32} & m_{33} & m_{34} \\
            m_{41} & m_{42} & m_{43} & m_{44} \\
        \end{pmatrix}.
    \end{equation*}
    First assume that the maximal observed submatrix $M_{:,234}$ has rank three. This implies that the second to the fourth columns of $M_E$ and thus the corresponding column sums are nonzero. We also assume that the sum of the observed entries in the first column is nonzero, because otherwise we can complete the missing entries to zero.
    Finally, we assume that $M_{34,234}$ has rank two. We consider the cases when the rank of $M_{34,234}$ is at most one or the rank of $M_{:,234}$ is at most two in~\Cref{ex:11_21_special_case}.

    We construct factors $A$ and $B$ such that $AB$ is a rank-$3$ completion of $M_E$.
    Define $A = M_{:,234}$ and 
    \begin{equation*}
        B_t = \begin{pmatrix}
            \vert & & \\
            b_1 & \operatorname{Id}_3 \\
            \vert
        \end{pmatrix} \in \mathbb{R}^{3 \times 4},
    \end{equation*}
    where $b_1 = (b_{11}, b_{21}, b_{31})^T \in \mathbb{R}^3$ is to be determined.
The restrictions on the column $b_1 \in \mathbb{R}^3$ are determined by the two observed entries, $m_{31}$ and $m_{41}$,  from the system of equations
    \begin{equation*}
        \begin{cases}
            a_3 b_1=m_{31},\\
            a_4 b_1=m_{41},
        \end{cases}
    \end{equation*}
    where $a_3$ and $a_4$ are the third and fourth rows of $A$, respectively. 

    If the $2 \times 2$ minor  $\det(M_{34,34})$ is nonzero, then
        \begin{align*}
        b_{11} &= t,\\
        b_{21} &= \frac{t \cdot \det(M_{34,24})-\det(M_{34,14})}{-\det(M_{34,34})},\\
        b_{31} & = \frac{t \cdot \det(M_{34,23})-\det(M_{34,13})}{\det(M_{34,34})},
    \end{align*}
    where $t \in \mathbb{R}$.
    Thus, in this case the first column of $B_t$ is of the form $(t,b_{21}(t),b_{31}(t))^T$ where $t$ is a parameter in $\mathbb{R}$.
    If $\det (M_{34,34}) = 0$, we may take the component of $b_1$ associated with another, nonvanishing $2 \times 2$ minor of $M_{34,234}$ as a parameter and solve the other two in terms of it.

    We will modify this factorization so that it will be easier to construct associated nested polytopes. 
    Denote by $s_i$ the $i$th column sum of the rank-$3$ completion $AB_t$ of $M_E$. We define
        \begin{equation*}
            \begin{split}
            &\hat A = A \begin{pmatrix}
             s_2& s_3& s_4 \\
             0 & 1 & 0 \\
            0 & 0 & 1 \\
        \end{pmatrix}^{-1} = \begin{pmatrix}
            \frac{m_{12}}{s_2} & m_{13} -\frac{m_{12}{s_3}}{s_2} & m_{14} -\frac{m_{12}{s_4}}{s_2}\\
            \frac{m_{22}}{s_2} &m_{23} -\frac{m_{22}{s_3}}{s_2} & m_{24} -\frac{m_{22}{s_4}}{s_2}\\
            \frac{m_{32}}{s_2} & m_{33} -\frac{m_{32}{s_3}}{s_2}& m_{34} -\frac{m_{32}{s_4}}{s_2}\\
            \frac{m_{42}}{s_2} & m_{43} -\frac{m_{42}{s_3}}{s_2}& m_{44} -\frac{m_{42}{s_4}}{s_2}\\
        \end{pmatrix} \ \text{ and }\\ 
        &\hat B_t = \begin{pmatrix}
             s_2& s_3& s_4 \\
             0 & 1 & 0 \\
            0 & 0 & 1 \\
        \end{pmatrix}B_t = \begin{pmatrix}
            t s_2 + b_{21}(t) s_3 + b_{31}(t) s_4 & s_2 & s_3 & s_4 \\
            b_{21}(t) & 0 & 1 & 0 \\
            b_{31}(t) & 0 & 0 & 1
        \end{pmatrix}.
            \end{split}
        \end{equation*}
    The geometry of the rank-$3$ completion of $M_E$ can be formulated via the nested cones $\hat P_t \subseteq \hat Q$, where $\hat P_t = \operatorname{cone}(\hat B_t)$ and $\hat Q = \{x \in \mathbb{R}^3 \mid \hat Ax \geq 0\}$.
    All four facets of $\hat Q$ are fixed, but only three rays of $\hat P_t$ are fixed. The fourth ray spanned by $(t s_2 + b_{21}(t) s_3 + b_{31}(t) s_4,b_{21}(t),b_{31}(t))^T$ depends on the parameter~$t$.

    We associate polytopes $P_t$ and $Q$ to $M_E$ by taking the affine slice $x=1$ of the cones $\hat P_t$ and $\hat Q$. 
    Then $P_t=\operatorname{conv}(p_1,p_2,p_3,p_4)$, with 
    \begin{align*}
        p_1 = \left( \frac{b_{21}(t)}{t s_2 + b_{21}(t) s_3 + b_{31}(t) s_4 }, \frac{b_{31}(t)}{t s_2 + b_{21}(t) s_3 + b_{31}(t) s_4 } \right), \ p_2 = \left( 0,0 \right), \ p_3 = \left( \frac{1}{s_3}, 0 \right), \ p_4 = \left( 0, \frac{1}{s_4} \right),
    \end{align*}
    and $Q = \{(1,x,y)  \mid \hat A(1,x,y) \geq 0 \}$. Denominators of $p_i$ are nonzero by the assumption that the sum of the observed entries in any column of $M_E$ is nonzero and the observation that $t s_2 + b_{21}(t) s_3 + b_{31}(t) s_4 = \frac{1}{\det (M_{34,34})}s_1(t)$, where $s_1(t)$ is the sum of the entries in the first column of the matrix $A B_t$.  
    For values of $t$ such that the completion is nonnegative we have $s_1(t) > 0$.
\end{example}

\begin{theorem}\label{thm:existence_of_nonnegative_rank-r_completion}
    Let $M_E \in \mathbb{R}^E_{\geq 0}$ be a partial nonnegative matrix.
    Then $M_E$ has a nonnegative rank-$r$ completion if and only if there exists a member $(\hat P,\hat Q)$ of the family of nested cones associated to $M_E$ such that there exists an $r$-simplicial cone $\hat \Delta$ satisfying $\hat P \subseteq \hat \Delta \subseteq \hat Q$. 
\end{theorem}

\begin{proof}
    Each member $(\hat P, \hat Q)$ of the family of cones associated to $M_E$ such that $\hat P \subseteq \hat Q$ corresponds to a distinct rank-$r$ completion which is nonnegative.
    There exists a member $(\hat{P},\hat{Q})$ in the family of nested cones such that there is an $r$-simplicial cone $\hat \Delta$ with $\hat{P} \subseteq \hat \Delta \subseteq \hat{Q}$ if and only if the corresponding completion has nonnegative rank $r$. 
\end{proof}

Finally, we consider a special case that was left open in~\Cref{example:11_and_21_missing}. For this we use the following lemma instead of a geometric characterization.

\begin{lemma}\label{lem:nonnegative_rank_of_submatrix_less_than_or_equal_k}
Let $M \in  \mathbb{R}^{p \times q}_{\geq 0}$. If $I \subset [p]$ such that $M_{I,:}$ has nonnegative rank at most $k$ and $p-|I| \leq r-k$, then $M$ has nonnegative rank at most $r$. Similarly, if $J \subset [q]$ such that $M_{:,J}$ has nonnegative rank at most $k$ and $q-|J| \leq r-k$, then $M$ has nonnegative rank at most $r$.
\end{lemma}

\begin{proof}
We prove the lemma for the case when $I \subset [p]$ such that $M_{I,:}$ has nonnegative rank at most $k$ and $p-|I| \leq r-k$. The proof for the second statement is analogous. Without loss of generality assume that $I=\{1,2,\ldots,|I|\}$. Since $M_{I,:}$ has nonnegative rank at most $k$, we can write $M_{I,:}=AB$, where $A \in \mathbb{R}^{|I| \times k}_{\geq 0}$ and $B \in \mathbb{R}^{k \times q}_{\geq 0}$. We define
\[
\tilde{A} = 
\left( 
\begin{array}{c|c} 
  A & 0 \\ 
  \hline 
  0 & \operatorname{Id}_{p-|I|} 
\end{array} 
\right)
\quad \text{and} \quad
\tilde{B} = 
\left( 
\begin{array}{c} 
  B \\ 
  \hline 
  M_{[p]\backslash I,:}
\end{array} 
\right),
\]
where $\operatorname{Id}_{p-|I|} $ is the $(p-|I|) \times (p-|I|)$ identity matrix.
Then $(\tilde{A},\tilde{B})$ gives a size at most $r$ nonnegative matrix factorization of $M$.
\end{proof}

\begin{example}\label{ex:11_21_special_case}
    Let $E = ([4] \times [4]) \setminus \{(1,1),(2,1) \}$ and consider the partial nonnegative matrices 
     \begin{equation*}
         M_E = \begin{pmatrix}
            ? & m_{12} & m_{13} & m_{14} \\
            ? & m_{22} & m_{23} & m_{24} \\
            m_{31} & m_{32} & m_{33} & m_{34} \\
            m_{41} & m_{42} & m_{43} & m_{44} \\
        \end{pmatrix}.
    \end{equation*}
    Here we consider the cases when $\rank (M_{:,234}) \leq 2$ or $\rank(M_{34,234}) \leq 1$ and characterize when such matrices have nonnegative rank at most three completions.
    If the maximal observed submatrix $M_{:,234}$ has rank at most two, then $M_E$ has a nonnegative rank at most three completion by \Cref{lem:nonnegative_rank_of_submatrix_less_than_or_equal_k}.

    Suppose then that the submatrix $M_{34,234}$ has rank at most one. 
    If $M_{34,:}$ has rank two and $M_{:,234}$ has rank three, then $M_E$ cannot completed to a rank-$3$ matrix.
    This can be proven using a similar argument to the one in the proof of \Cref{lemma:removing-one-row-and-one-column}.
    Therefore, we need to only consider cases when $\rank(M_{34,:})\leq 1$ or $\rank (M_{:,234}) \leq 2$.
    In both cases, the matrix $M_E$ can be completed to a nonnegative rank at most three matrix by \Cref{lem:nonnegative_rank_of_submatrix_less_than_or_equal_k}.
    
\end{example}

\subsection{Entries \texorpdfstring{$(1,1)$ and $(2,1)$}{(1,1) and (2,1)} missing} \label{sec:missing_entries_11_and_21}

    In this subsection, we study the geometry of nonnegative rank-$3$ completions of partial nonnegative $4 \times 4$ matrices with the entries $(1,1)$ and $(2,1)$ missing.
    By~\Cref{lem:properties_of_geom_11_12}, rank-$3$ completions of $M_E$ correspond to a family of nested polytopes $P_t \subseteq Q$ in $\mathbb{R}^2$, where the outer polygon $Q$ and three vertices of the inner polygon $P_t$ are fixed and the last vertex of $P_t$ lies on a line that passes through one vertex of $Q$.
    In \Cref{lem:11_12_sufficient_condition} we give sufficient conditions to ensure that there exists a nonnegative rank-$3$ completion for such a partial matrix, and in \Cref{ex:square} we give an example of a partial nonnegative matrix with a rank-$3$ completion that is nonnegative but with no nonnegative rank-$3$ completions.

\begin{lemma}\label{lem:properties_of_geom_11_12}
    Let $E = [4] \times [4] \setminus \{ (1,1), (2,1) \}$ and $M_E \in \mathbb{R}_{\geq 0}^E$ be a partial matrix with no zero rows or columns, 
    and assume that the $2 \times 2$ minor $\det(M_{34,34})$ does not vanish.  
    Let $P_t$ and $Q$ be defined as in \Cref{example:11_and_21_missing}.
    Then, the vertex $p_1$ of $P_t$ always lies on a line that goes through the vertex of $Q$ defined by the third and fourth defining facets of $Q$.
    The nonnegative rank-$3$ completions of $M_E$ correspond to those points $p_1$ on the line such that there exists a triangle $\Delta$ such that $P_t \subseteq \Delta \subseteq Q$.
\end{lemma}

\begin{proof}
    The vertex $p_1$ of $P_t$ is of the form $\left( \frac{b_{21}(t)}{t s_2 + b_{21}(t) s_3 + b_{31}(t) s_4 }, \frac{b_{31}(t)}{t s_2 + b_{21}(t) s_3 + b_{31}(t) s_4 } \right)$.
    One can check by elimination, that for every $t \in \mathbb{R}$, the vertex $p_1$ lies on the line
    \begin{equation}\label{eq:p1_line}
        l_{p_1} =-\det(M_{34,12}) + c_x x + c_y y,
    \end{equation}
    where
    \begin{align*}
        c_x &= m_{41}
        (\det(M_{13,23}) + \det(M_{23,23})  - 
       \det(M_{34,23}))-m_{31}(\det(M_{14,23}) + 
       \det(M_{24,23}) + \det(M_{34,23})), \\
       c_y &= m_{41}(\det(M_{13,24}) + \det(M_{23,24})  - 
       \det(M_{34,24}))-m_{31}(\det(M_{14,24}) + 
       \det(M_{24,24}) + \det(M_{34,24})).
    \end{align*}
    In particular, this line depends on the entries $m_{31}$ and $m_{41}$.

    Let $q_1$ be the intersection point of the third and fourth facet defining $Q$, that is, $\hat a_3^T (1,q_1)= \hat a_4^T (1,q_1) = 0$.
    The line $l_{p_1}$ always goes through the point $q_1$ as $l_{p_1}(q_1) = 0$. However, $p_1 = q_1$ only when $m_{31} = m_{41} = 0$.
    
    The points on $l_{p_1}$ inside $Q$ correspond to different rank-$3$ completions of $M_E$ which are nonnegative.
    Whenever there exists a triangle $\Delta$ such that $P_t \subseteq \Delta \subset Q$ the corresponding completion of $M_E$ has nonnegative rank three by~\Cref{thm:existence_of_nonnegative_rank-r_completion}.
\end{proof}

In the next lemma we give sufficient conditions for a partial matrix $M_E$ to have a nonnegative rank-$3$ completion.

\begin{lemma}\label{lem:11_12_sufficient_condition}
    Let $M_E, P_t$, and $Q$ be defined as above.
    Let $l_{p_1}$ be the line defined in \eqref{eq:p1_line} on which the last vertex of $P_t$ lies. 
    If $l_{p_1}$ intersects $\operatorname{conv}(p_2,p_3,p_4)$ or the lines containing the edges of $\operatorname{conv}(p_2,p_3,p_4)$ inside $Q$, then $M_E$ has a nonnegative rank-$3$ completion.
\end{lemma}

\begin{proof}
    Suppose first that the line $l_{p_1}$ intersects $\operatorname{conv}(p_2,p_3,p_4)$.
    By \Cref{lem:properties_of_geom_11_12} each point of $l_{p_1}$ inside $Q$ corresponds to a rank-$3$ completion of $M_E$ which is nonnegative.
    Let $p_1 \in l_{p_1} \cap \operatorname{conv}(p_2,p_3,p_4)$. 
    Then, $P_t = \operatorname{conv}(p_1,p_2,p_3,p_4) = \operatorname{conv}(p_2,p_3,p_4)$, where $\operatorname{conv}(p_2,p_3,p_4)$ is a triangle.
    Hence, the completion of $M_E$ corresponding to this $p_1$ has nonnegative rank three.

    Suppose now that $l_{p_1}$ intersects a line which contains an edge of $\operatorname{conv}(p_2,p_3,p_4)$ at a point $p_1$ outside of $\operatorname{conv}(p_2,p_3,p_4)$.
    Without loss of generality, assume that the line contains the edge $\overline{p_2p_3}$ of $\operatorname{conv}(p_2,p_3,p_4)$.
    Then either $\overline{p_2 p_3} \subset \overline{p_1p_2}$ or $\overline{p_2 p_3} \subset \overline{p_1p_3}$. In the first case, $\operatorname{conv}(p_2,p_3,p_4) \subset \operatorname{conv}(p_1,p_2,p_4)$, and in the second case $\operatorname{conv}(p_2,p_3,p_4) \subset \operatorname{conv}(p_1,p_3,p_4)$. In both cases, the polytope $\operatorname{conv}(p_1,p_2,p_3,p_4)$ is a triangle.
    Hence, the completion of $M_E$ corresponding to this $p_1$ has nonnegative rank three.
\end{proof}

\begin{example}\label{ex:square}
    In this example we construct a partial nonnegative matrix which has a rank-$3$ completion that is nonnegative but no nonnegative rank-$3$ completions.
Consider partial matrices of the form
\begin{equation*}
    M_E = \begin{pmatrix}
        ? & 5 & 1 & 9\\
        ? & 1 &  7 & 7\\
        m_{31} & 5 & 9 & 1\\
        0 & 9 & 3 & 3
    \end{pmatrix}
\end{equation*}
where $m_{31} \geq 0$.
The rank-$3$ completions of $M_E$ are of the form 
\begin{equation*}
    M_{t} = \begin{pmatrix}
        t-m_{31} & 5 & 1 & 9\\
        t & 1 &  7 & 7\\
        m_{31} & 5 & 9 & 1\\
        0 & 9 & 3 & 3
    \end{pmatrix}.
\end{equation*}
The completion $M_t$ is nonnegative for $t \geq m_{31}$.

The polytopes corresponding to $M_t$ are $Q$ and $P_t$ constructed as in \Cref{example:11_and_21_missing}.
The outer polytope $Q$ is defined by the inequalities
\begin{align*}
    x -y + \frac{1}{16} \geq 0, \ 
    -x + y+ \frac{1}{16} \geq 0, \
    x + y + \frac{1}{120}\geq 0, \ \text{ and } \
    x + y -\frac{3}{40} \leq 0.
\end{align*}
Equivalently $Q = \operatorname{conv}(q_1,q_2,q_3,q_4)$ with
\begin{align*}
    q_1 = \left(\frac{1}{160},\frac{11}{160}\right), \ q_2 = \left( \frac{11}{160}, \frac{1}{160} \right)  , \ q_3 = \left(\frac{13}{480}, -\frac{17}{480} \right), \ \text{ and } \ q_4 = \left( -\frac{17}{480}, \frac{13}{480} \right).
\end{align*}
The inner polytope $P_t$ is defined as $P_t = \operatorname{conv}(p_1,p_2,p_3,p_4)$ where  
\begin{equation*}
        p_1=\left(\frac{10m_{31}+t}{160t}, \frac{11}{160} -\frac{m_{31}}{16t}\right), \ p_2 = \left(0, 0\right), \ p_3 = \left(\frac{1}{20},0\right), \text{ and } \ p_4 =\left(0,\frac{1}{20}\right).
\end{equation*}
We notice that the inner vertex $p_1$ is always on the edge $\overline{q_1q_2}$ of $Q$ which is defined by the equation $x + y -\frac{3}{40} = 0$.
Moreover, $p_1 = q_1$ only when $m_{31}=0$.
Otherwise, $p_1 \rightarrow q_1$ as $t \rightarrow \infty$ but $p_1 \neq q_1$.

We show that there does not exist a triangle $\Delta$ such that $P_t\subseteq \Delta \subseteq Q$.
Since $p_1$ is always on the edge $\overline{q_1q_2}$ any triangle nested between $P_t$ and $Q$ must have one vertex, namely $p_1$, or two vertices on this edge.

First consider the triangles  such that $p_1$ is the only point of the triangle on the edge $\overline{q_1q_2}$. 
Let $v_1, v_2,w$ be the intersection points of the boundary of $Q$ and the lines $\overline{p_1p_4}$, $\overline{v_1p_2}$, and $\overline{v_2p_3}$, respectively.
There is a triangle $\Delta$ nested between $P_t$ and $Q$ if and only if the lines $\overline{p_1v_1}$ and $\overline{v_2 p_3}$ intersect inside $Q$. To show that there is no such triangle, we study the movement of the point $w$ as $p_1$ moves along the edge $\overline{q_1 q_2}$.

\begin{figure}[htbp!] 
    \centering
    \begin{subfigure}[b]{0.45\linewidth}
        \centering
        \includegraphics[width=\linewidth]{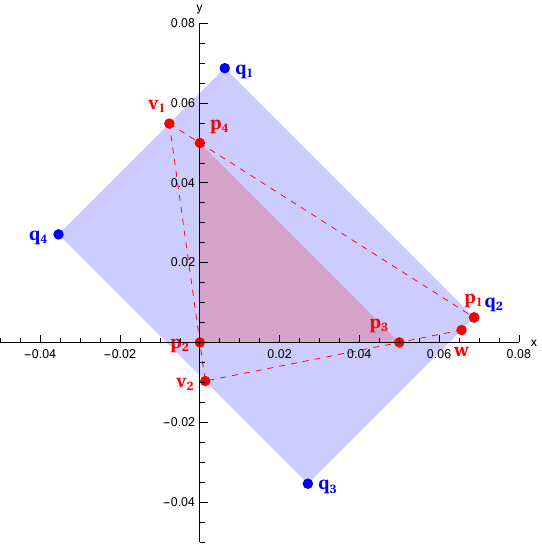} 
        \caption{Configuration 1}
        \label{subfig:conf1}
    \end{subfigure}
    \hfill 
    \begin{subfigure}[b]{0.45\linewidth}
        \centering
        \includegraphics[width=\linewidth]{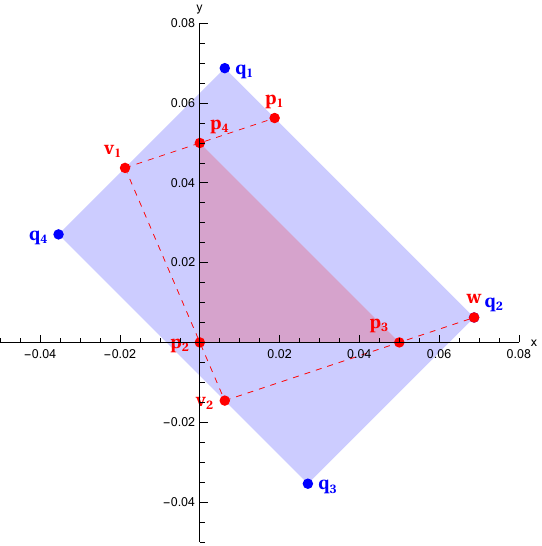}
        \caption{Configuration 2}
        \label{subfig:conf2}
    \end{subfigure}
    \par\vspace{0.5em} 
    \begin{subfigure}[b]{0.45\linewidth}
        \centering
        \includegraphics[width=\linewidth]{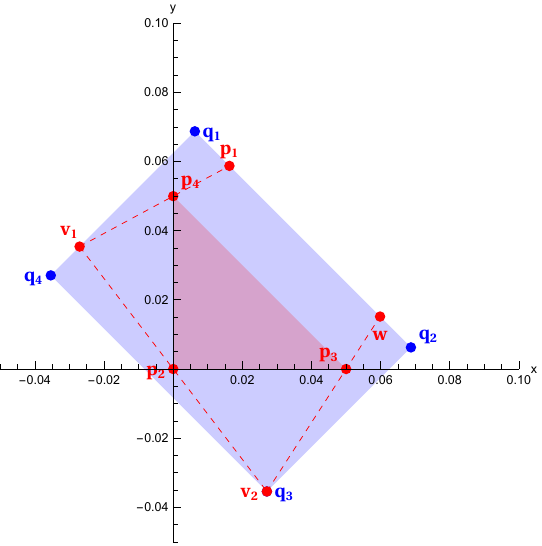}
        \caption{Configuration 3}
        \label{subfig:conf3}
    \end{subfigure}
    \hfill
    \begin{subfigure}[b]{0.45\linewidth}
        \centering
        \includegraphics[width=\linewidth]{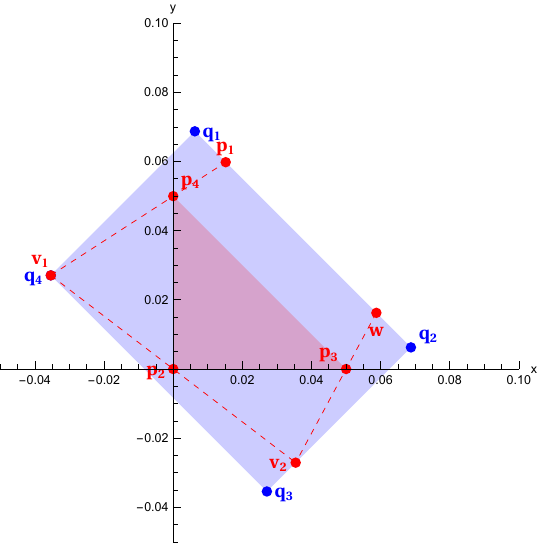}
        \caption{Configuration 4}
        \label{subfig:conf4}
    \end{subfigure}
    \begin{subfigure}[b]{0.45\linewidth}
        \centering
        \includegraphics[width=\linewidth]{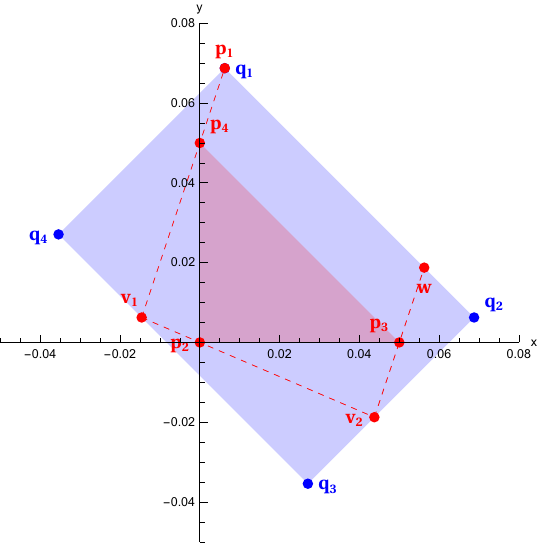}
        \caption{Configuration 5}
        \label{subfig:conf5}
    \end{subfigure}
    \caption{Critical configurations in \Cref{ex:square}.} 
    \label{fig:square_critical}
\end{figure}

As long as each of $v_1$, $v_2$, and $w$ moves along a fixed edge, moving $p_1$ from $q_2$ to $q_1$ causes $v_1$, $v_2$, and $w$ to move monotonically along their respective edges. Consequently, it suffices to consider only the situations in which $p_1$ lies at an endpoint of $\overline{q_1 q_2}$ or one of $v_1$, $v_2$, or $w$ changes edges. We refer to such situations as critical configurations. If, in any critical configuration, $w$ is positioned clockwise from $p_1$, then the same holds for every position of $p_1$ on $\overline{q_1 q_2}$, since $w$ moves monotonically as $p_1$ moves from $q_2$ to $q_1$. In that case, there is no triangle $\Delta$ nested between $P_t$ and $Q$ for any position of $p_1$ along the edge $\overline{q_1 q_2}$.

Suppose first that $p_1 = q_2$. In this case $w$ lies on the edge $\overline{q_2 q_3}$ as depicted in \Cref{subfig:conf1}. 
As the vertex $p_1$ moves along the edge $\overline{q_1q_2}$, the next critical configuration happens when the vertex $w$ moves from the facet $\overline{q_2q_3}$ to $\overline{q_1q_2}$ of $Q$. 
This happens when $p_1 = \left( \frac{3}{160}, \frac{9}{160} \right)$ and is depicted in \Cref{subfig:conf2}.
The second of these critical configurations happens when $p_1 = \left( \frac{13}{800}, \frac{47}{800} \right)$; at this point the vertex $v_2$ of $\Delta$ moves from the edge $\overline{q_3q_4}$ of $Q$ to the edge $\overline{q_2q_3}$.
This is depicted in \Cref{subfig:conf3}.
In \Cref{subfig:conf4}, the next configuration is depicted; when the point $p_1 = \left(\frac{17}{1120} , \frac{67}{1120} \right)$ is reached, the vertex $v_1$ moves from the edge $\overline{q_1q_4}$ to $\overline{q_3q_4}$.
And lastly, the configuration when $p_1$ reaches the vertex $q_1$ of $Q$ is depicted in \Cref{subfig:conf5}.
In each of these cases, the lines $\overline{p_1v_1}$ and $\overline{v_2 p_3}$ intersect outside $Q$ and therefore there is not triangle nested between $P_t$ and $Q$.

Next, we consider the triangles $\Delta$ such that more than one point of $\Delta$ is on the edge $\overline{q_1q_2}$.
In particular, a line segment of the edge $\overline{q_1q_2}$ is an edge of $\Delta$.
If there is a triangle nested between $P_t$ and $Q$ with a line segment on the edge $\overline{q_1q_2}$, then there is a triangle nested between $P_t$ and $Q$ with the entire edge $\overline{q_1q_2}$ of $Q$ being an edge of $\Delta$. 
If there is a triangle  nested between $P_t$ and $Q$ with the entire edge $\overline{q_1q_2}$ of $Q$ being an edge of $\Delta$, then there is a triangle with edges containing $\overline{q_1p_4}$ and $\overline{q_2p_3}$. The smallest triangle corresponding to this case is depicted in \Cref{fig:square_full_edge}, which is not contained in $Q$.

\begin{figure}[htbp!]
    \centering
    \includegraphics[width=0.45\linewidth]{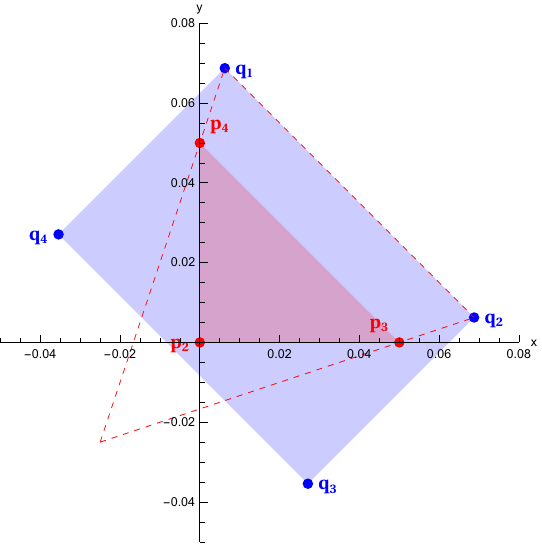}
    \caption{The triangle with $\overline{q_1q_2}$ as an edge.}
    \label{fig:square_full_edge}
\end{figure}

Therefore, we conclude that there does not exist a triangle $\Delta$ nested between $P_t$ and $Q$ regardless of the position of the vertex $p_1$.
Thus, no rank-$3$ completion of $M_E$ to a nonnegative matrix is of nonnegative rank-$3$.
\end{example}

\subsection{Entries \texorpdfstring{$(1,1)$ and $(2,2)$}{(1,1) and (2,2)} missing} \label{sec:missing_entries_11_and_22}

    In this subsection, we study the geometry of nonnegative rank-$3$ completions of partial nonnegative $4 \times 4$ matrices with the entries $(1,1)$ and $(2,2)$ missing.
    We show that rank-$3$ completions of $M_E$ correspond to a family of nested polytopes $P_t \subseteq Q_t$ in $\mathbb{R}^2$, where three facets of the outer polygon $Q_t$ and three vertices of the inner polygon $P_t$ are fixed and  the movement of the remaining facet of $Q_t$ depends on the movement of the remaining vertex of $P_t$. In~\Cref{prop:11_and_22_sufficient_condition}, we give sufficient conditions for the existence of nonnegative rank-$3$ completions for such partial matrices. In \Cref{ex:11_22_counter_example} we give an example of a partial nonnegative matrix with a rank-$3$ completion that is nonnegative but with no nonnegative rank-$3$ completions.

Consider a partial nonnegative matrix
    \begin{equation*}
        M_E = \begin{pmatrix}
            ? & m_{12} & m_{13} & m_{14} \\
            m_{21} & ? & m_{23} & m_{24} \\
            m_{31} & m_{32} & m_{33} & m_{34} \\
            m_{41} & m_{42} & m_{43} & m_{44} \\
        \end{pmatrix} \in \mathbb{R}_{\geq 0}^{E},
    \end{equation*}
    where the maximal observed submatrices $M_{134,234}$ and $M_{234,134}$ have rank three.
We will construct a rank-$3$ factorization $AB$ of a rank-$3$ completion of $M_E$.
Define
\begin{equation} \label{eq:11_22_AB}
    A=\begin{pmatrix}
        a_{11} & a_{12} & a_{13} \\
        1 & 0 & 0 \\
        0& 1 & 0 \\
        0& 0 & 1
    \end{pmatrix} \ \text{ and } \ B= \begin{pmatrix}
        m_{21} & b & m_{23} & m_{24} \\
        m_{31} & m_{32} & m_{33} & m_{34} \\
        m_{41} & m_{42} & m_{43} & m_{44} \\
    \end{pmatrix},
\end{equation}
where $a_{11},a_{12},a_{13},b \in \mathbb{R}$ are to be determined.
We assume that $\rank(M_{234,34})=2$ and $\rank(M_{34,234})=2$.

    The first row $a_1=(a_{11},a_{12},a_{13})$ of $A$ and the second column $b_2=(b, m_{32},m_{42})$ of $B$ satisfy the system of equations
    \begin{equation*}
        \begin{cases}
            a_1b_2 = m_{12},\\
            a_1b_3 = m_{13},\\
            a_1b_4 = m_{14}.
        \end{cases}  
    \end{equation*}
    We can parametrize the solutions to this system in terms of a parameter $t \in \mathbb{R}$:
    \begin{equation} \label{eqn:solution_11_22}
        \begin{split}
            &a_{11}(t) = \frac{\det(M_{134,234})}{m_{23}\det (M_{34,24})-t \det (M_{34,34})-m_{24}\det (M_{34,23})},\\
            &a_{12}(t) = \frac{t\det (M_{14,34})-m_{42} \det (M_{12,34})-m_{12}\det (M_{24,34})}{-(m_{23}\det (M_{34,24})-t \det (M_{34,34})-m_{24}\det (M_{34,23}))},\\
            &a_{13}(t) = \frac{t\det (M_{13,34})-m_{32} \det (M_{12,34})-m_{12}\det (M_{24,34})}{m_{23}\det (M_{34,24})-t \det (M_{34,34})-m_{24}\det (M_{34,23})},\\
            &b(t)=t,
        \end{split}
    \end{equation}
    for $t$ such that $m_{23}\det (M_{34,24})-t \det (M_{34,34})-m_{24}\det (M_{34,23}) \neq 0$.
    The rank-$3$ completions of $M_E$ are given by $M_t:=A_tB_t$, where $A_t$ and $B_t$ are obtained from $A$ and $B$ in~\eqref{eq:11_22_AB} by substituting $a_{11},a_{12},a_{13},b$ by $a_{11}(t),a_{12}(t),a_{13}(t),b(t)$ in~\eqref{eqn:solution_11_22}.

    Whenever $M_t=A_t B_t \geq 0$, the factorization $A_t B_t$ gives a pair of nested cones $\hat Q_t=\{x \in \mathbb{R}^3:A_t x \geq 0\}$ and $\hat P_t=\operatorname{cone}(B_t)$. By considering these pairs of nested cones for all $t \in \mathbb{R}$ such that $A_t B_t \geq 0$, we get a family of pairs of nested cones. All pairs in this family share three facets of $\hat Q_t$ and three rays of $\hat P_t$.
    The movement of the remaining facet of $\hat Q_t$ depends on the movement of the remaining ray of $\hat P_t$.

        In the above construction we assumed that $M_{234,34}$ and $M_{34,234}$ have rank two.
    If $\rank(M_{:,34})=\rank(M_{34,:})=2$, then one may use a similar construction as above and express $a_{11},a_{12},a_{13},b$ in terms of a parameter $t \in \mathbb{R}$.

    On the other hand, if $\rank(M_{:,34}) \leq 1$ or $\rank(M_{34,:}) \leq 1$, the matrix $M_E$ has a nonnegative rank at most three completion by \Cref{lem:nonnegative_rank_and_rank_equal_when_1_2} and \Cref{lem:nonnegative_rank_of_submatrix_less_than_or_equal_k}.

\begin{proposition} \label{prop:11_and_22_sufficient_condition}
    Let $E = [4] \times [4] \setminus \{(1,1),(2,2) \}$.
    Let $M_E \in \mathbb{R}_{\geq 0}^E$ be a partial nonnegative matrix with no zero rows or columns, and let $M_t =A_tB_t$ be a rank-$3$ nonnegative completion of $M_E$ where $A_t$ and $B_t$ are as above.
    Then $M_E$ has a nonnegative rank-$3$ completion if there exists $t \geq 0$ such that $M_t$ is nonnegative and 
    \begin{enumerate}
        \item  $a_{1j}(t) \geq 0$ for all $j=1,2,3$, or
        \item  among the three values $a_{11}(t), a_{12}(t), a_{13}(t)$, two are negative and one is positive, or 
        \item  among the three values $a_{11}(t), a_{12}(t), a_{13}(t)$, one is negative, one is zero, and one is positive.
    \end{enumerate}
\end{proposition}

\begin{proof}
Since $A_{234,:}$ is the $3 \times 3$ identity matrix, $\hat Q_t=\{x \in \mathbb{R}^3:A_t x \geq 0\}$ is the subset of $\mathbb{R}^3_{\geq 0}$ defined by the inequality $a_1(t) x \geq 0$. We will characterize the cases when the cone $\hat Q_t$ is simplicial. Then $\hat Q_t$ itself can be taken as the simplicial cone between $\hat P_t$ and $\hat Q_t$ and hence the nonnegative rank of $M_t$ is three by \Cref{thm:existence_of_nonnegative_rank-r_completion}. Since $M_t$ has rank three by assumption, $\hat Q_t$ is 3-dimensional. 

The first possibility is that the inequality $a_1(t) x \geq 0$ is redundant, i.e., $\hat Q_t=\mathbb{R}^3_{\geq 0}$. This corresponds to the first case when $a_{1j}(t) \geq 0$ for all $j=1,2,3$.

If $a_1(t) x \geq 0$ is irredundant, then two rays of $\hat Q_t$ lie on the hyperplane $a_1(t) x = 0$ (at least one of which is not a ray of $\mathbb{R}^3_{\geq 0}$) and there is one more ray that is also a ray of $\mathbb{R}^3_{\geq 0}$. If the two rays of $\hat Q_t$ that lie on the hyperplane $a_1(t) x = 0$ are not rays of $\mathbb{R}^3_{\geq 0}$, then among the three values $a_{11}(t), a_{12}(t), a_{13}(t)$, two are negative and one is positive. Otherwise, one ray of $\hat Q_t$ that lies on the hyperplane $a_1(t) x = 0$ is also a ray of $\mathbb{R}^3_{\geq 0}$ in which case among the three values $a_{11}(t), a_{12}(t), a_{13}(t)$, one is negative, one is zero, and one is positive. 
\end{proof}

\begin{example}\label{ex:11_22_counter_example}
    This example shows that there exists a partial nonnegative matrix $M_E \in \mathbb{R}^{E}$ with the missing entries $(1,1)$ and $(2,2)$ which does not have a nonnegative rank-$3$ completion, although it has a rank-$3$ completion with nonnegative entries.

    Consider the partial matrix
    \begin{equation*}
        M_E= \begin{pmatrix}
        ? & 1  &  112/425  &  1 \\
        1/10 &  ? &  1/100 &  7/20\\
        1/10 & 10 & 9/10 & 3/5 \\
        4/5 &  9 & 9/100 &  1/20
    \end{pmatrix}
    \end{equation*}
    A rank-$3$ completion of $M_E$ with nonnegative entries is given by $A_t B_t$, where
    \begin{equation*}
        A_t = \begin{pmatrix}
            \frac{100601}{42007 + 153 t} & \frac{12055 + 1306t}{42007 + 153 t} & -\frac{3 (2909 + 4204 t)}{17 (2471 + 9 t)} \\
            1&0&0\\
            0&1&0\\
            0&0&1
        \end{pmatrix} \ \text{ and } \ B_t= \begin{pmatrix}
            1/10 &  t &  1/100 &  7/20\\
        1/10 & 10 & 9/10 & 3/5 \\
        4/5 &  9 & 9/100 &  1/20
        \end{pmatrix}
    \end{equation*}
    with $0 \leq t \leq \frac{4284}{9959}$.
    The family $(A_t,B_t)$ defines nested polytopes $P_t \subseteq Q_t$ in $\mathbb{R}^2$ by taking the cones $\operatorname{cone}(B_t) \subseteq \{x \in \mathbb{R}^3 \mid A_tx\geq 0\} \subseteq \mathbb{R}^3$ and considering the affine slice of the cone defined by the plane $x+y+z=1$.
    The outer polytope $Q_t$ is given by 
    \begin{align*}
        Q_t = \{(x,y) \in \mathbb{R}^2 \mid A_t(x,y,-x-y+1) \geq 0\}
    \end{align*}
    with edges given by 
    \begin{align*}
    &\frac{-8727-12612t + (109328 +12612t) x + (20782 +13918t) y}{17 (2471 + 9 t)} \geq 0,\\
        &x \geq 0, y\geq 0, 1-x-y \geq 0.
    \end{align*}
    Similarly, the inner polytope $P_t$ is $P_t = \operatorname{conv}(p_1,p_2,p_3,p_4)$, where 
    \begin{align*}
        &p_1 = \left( \frac{1}{10}, \frac{1}{10} \right), \ p_2= \left( \frac{t}{19 + t}, \frac{10}{19+t}\right), \\ 
        &p_3 = \left( \frac{1}{100}, \frac{9}{10} \right), \ p_4 = \left( \frac{7}{20}, \frac{3}{5} \right).
    \end{align*}
    The polytopes $P_0$, $P_c$, $Q_0$ and $Q_c$ where $c= \frac{4284}{9959}$ are depicted in \Cref{fig:11_22_boundary_polytopes}.
    \begin{figure}[htbp!]
    \centering
    \includegraphics[width=0.5\linewidth]{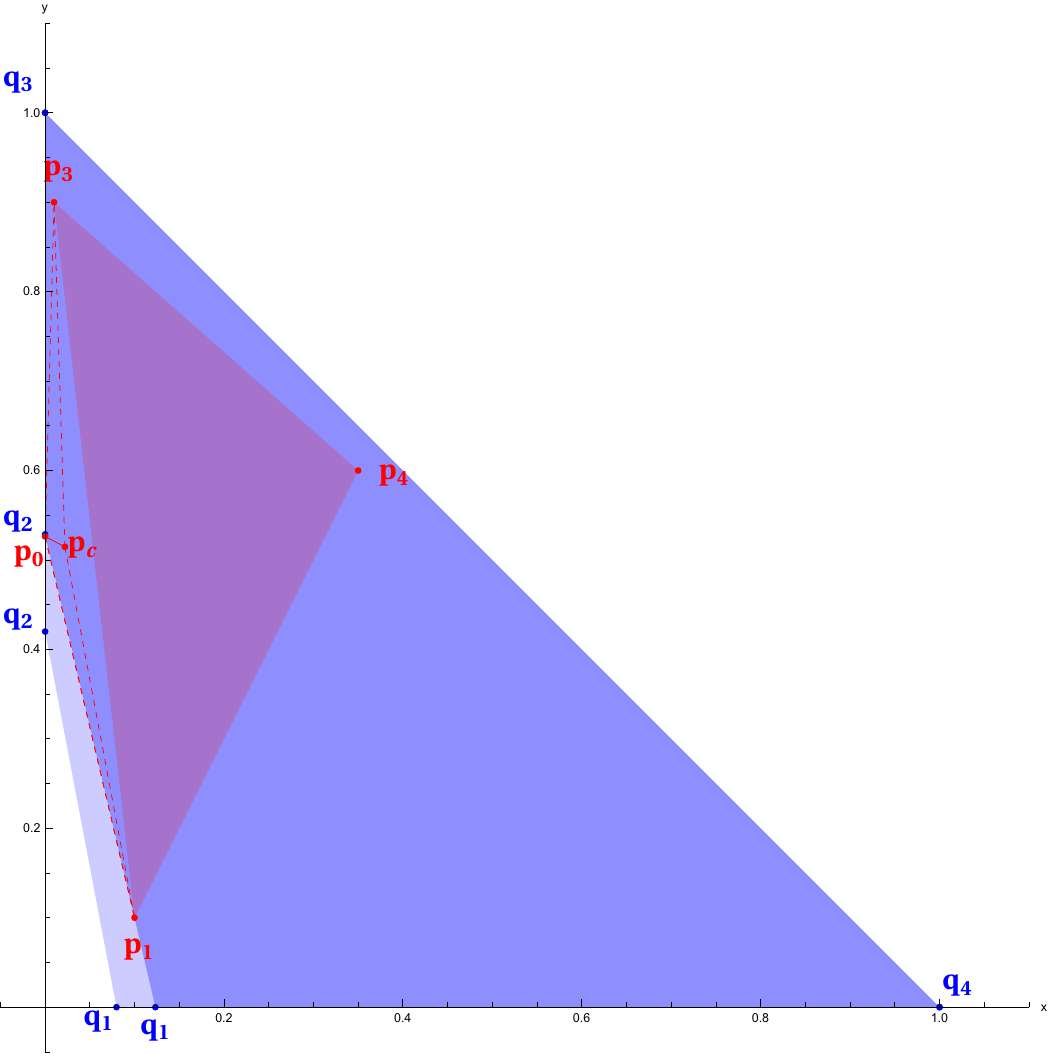}
    \caption{The polytopes $P_t$ (in red) and $Q_t$ (in blue) from \Cref{ex:11_22_counter_example} for the values $t=0$ and $t=c$. 
    The moving edge of $Q_t$ is the edge $\overline{q_1q_2}$, and the moving vertex of $P_t$ is the vertex $p_2$, which moves along the line segment $\overline{p_0p_c}$. 
    The polytope $Q_0$ (in lighter blue) always contains the polytope $Q_c$ (in darker blue).
    Similarly, the polytope $P_0$ always contains the polytope $P_c$.}
    \label{fig:11_22_boundary_polytopes}
    \end{figure}

    We show that for all $0 \leq t \leq c$ there does not exist a triangle $\Delta$ such that $P_t \subseteq \Delta \subseteq Q_t$.
    Assume the contrary and that there exists a $0 \leq \tilde t \leq c$ such that there exists a triangle $\Delta$ nested between $P_{\tilde t}$ and $Q_{\tilde t}$, that is $P_{\tilde t} \subseteq \Delta \subseteq Q_{\tilde t}$.
    Since for all $0 \leq t_1 \leq t_2 \leq c$ we have $P_{t_2} \subseteq P_{t_1}$ and $Q_{t_2} \subseteq Q_{t_1}$, it follows that in particular $P_c \subseteq P_{\tilde t}$ and $Q_{\tilde t} \subseteq Q_0$.
    Hence, the triangle $\Delta$ is such that $P_c \subseteq \Delta \subseteq Q_0$.
    However, using the geometric algorithm in  \Cref{cor:rank-3_geometric_algorithm} we can computationally show that there does not exist a triangle nested between $P_c$ and $Q_0$.
    The computations can be found in the git repository of the project.
    Therefore, there does not exist $t \in [0,c]$ such that there is a triangle $\Delta$ satisfying $P_t \subseteq \Delta \subseteq Q_t$.
\end{example}

\bibliographystyle{alpha}
\bibliography{references}

\end{document}